\documentclass[12pt]{amsart}

\usepackage{amsmath,bm,amssymb,amsthm,textcomp}
\usepackage{enumitem}
\usepackage{mathtools}
\usepackage{hyperref}
\usepackage[numbers,sort&compress]{natbib}
\usepackage{tikz-cd}
\usetikzlibrary{cd}
\DeclareMathOperator{\id}{id}
\DeclareMathOperator{\realizable}{re}
\DeclareMathOperator{\Proj}{Proj}
\DeclareMathOperator{\hdim}{\dim_H}
\DeclareMathOperator{\bdim}{\dim_B}
\DeclareMathOperator{\ubdim}{\overline{\dim}_B}
\DeclareMathOperator{\lbdim}{\underline{\dim}_B}
\DeclareMathOperator{\covdim}{\dim_{cov}}

\theoremstyle{plain}
\newtheorem{theorem}{Theorem}[section]
\newtheorem{lemma}[theorem]{Lemma}
\newtheorem{proposition}[theorem]{Proposition}
\newtheorem{corollary}[theorem]{Corollary}
\newtheorem*{observation}{Observation}

\theoremstyle{remark}
\newtheorem{remark}[theorem]{Remark}

\counterwithin{equation}{section}

\allowdisplaybreaks

\begin{document}

\title{On the error-sum function of Pierce expansions}

\author{Min Woong Ahn}
\address{Department of Mathematics, SUNY at Buffalo, Buffalo, NY 14260-2900, USA}
\email{minwoong@buffalo.edu}

\date{October 8, 2023}


\begin{abstract}
We introduce the error-sum function of Pierce expansions. Some basic properties of the error-sum function are analyzed. We also examine the fractal property of the graph of it by calculating the Hausdorff dimension, the box-counting dimension, and the covering dimension of the graph.
\end{abstract}

\maketitle

\tableofcontents

\section{Introduction} \label{Section 1}

The notion of the error-sum function was first studied by Ridley and Petruska \cite{RP00} in the context of regular continued fraction expansion. For any real number $x$, the error-sum function of the continued fraction expansion is defined by
\[
P (x) \coloneqq \sum_{n=0}^\infty q_n(x) \left( x - \frac{p_n(x)}{q_n(x)} \right),
\]
where
\[
\frac{p_n(x)}{q_n(x)} \coloneqq [a_0(x); a_1(x), a_2(x), \dotsc, a_n(x)] \coloneqq a_0(x) + \cfrac{1}{a_1(x) + \cfrac{1}{a_2(x) + \cfrac{1}{\ddots + \cfrac{1}{a_n(x)}}}}
\]
is the $n$th convergent (or approximant) of the continued fraction expansion, with $p_n(x)$, $q_n(x)$ coprime, $a_0(x)$ an integer, and $a_1(x), \dotsc, a_n(x)$ positive integers. For any rational number $x$, $a_k(x)$ are undefined for some point and on, and hence $P(x)$ is a series of finitely many terms. In such case, $x = [a_0(x); a_1(x), \dotsc, a_n(x)]$ for some $n \geq 0$ and if, further, $n \geq 1$ then $a_n(x)>1$. In fact, Petruska \cite{Pet92} used the error-sum function to prove the existence of a $q$-series $F(z) = 1 + \sum_{n=1}^\infty \left( \prod_{k=1}^n (A-q^k) \right) z^n$ with the radius of convergence $R$ for arbitrary given $R>1$. Here, $A = e^{2 \pi i P(\beta)}$ and $q = e^{2 \pi i \beta}$, where  $\beta$ is some irrational number satisfying certain conditions in terms of the $q_n(\beta)$. Moreover, there are a number of studies using error-sum functions to obtain number-theoretical results. See, e.g., \cite{AB16, Els11, Els14, ES11, ES12} for further applications of error-sum functions.

The continued fraction expansion, along with the decimal expansion, is one of the most famous representations of a real number. Since there is, as is well known, a wide range of representations of real numbers (see \cite{Gal76} and \cite{Sch95} for details), it was natural for intrigued researchers to define the error-sum function for other types of representations and investigate its basic properties. To name but a few, the error-sum functions were defined and studied in the context of the tent map base series \cite{DT11}, the classical L{\"u}roth series \cite{SW07}, the alternating L{\"u}roth series \cite{SMZ06}, and the $\alpha$-L{\"u}roth series \cite{CWY14}. In the previous studies, the list of examined basic properties includes, but is not limited to, boundedness, continuity, integrality, and intermediate value property (or Darboux property) of the error-sum function, and the Hausdorff dimension of the graph of the function. 

The {\em Pierece expansion} is another classical representation of real numbers introduced by Pierce \cite{Pie29} about a century ago. Since then, a number of studies were conducted to study the arithmetic and metric properties of the Pierce expansions. See, e.g., \cite{DDF22, ES91, Fan15, PVB98, Rem54, Sha84, Sha86, Sha93, Var17, VBP99}. It has proven to be useful in number theory, and we mention two applications among others. Firstly, the Pierce expansion provides us with a simple irrationality proof of a real number (see \cite[p.~24]{Sha86}). {\em A real number has an infinite Pierce expansion if and only if it is irrational}. For instance, the irrationalities of $1-e^{-1}$, $\sin 1$, and $\cos 1$ follow, respectively, from their infinite Pierce expansions which coincide with their usual series expansions obtained from the Maclaurin series. As for the other application, Varona \cite{Var17} constructed transcendental numbers by means of Pierce expansions.

Although the Pierce expansion has a long history and is widely studied, different from other representations mentioned above, its error-sum function has not yet been studied. In this paper, we define the error-sum function of the Pierce expansion, and analyze its basic properties and fractal properties of its graph.

The paper is organized as follows. In Section \ref{Section 2}, we introduce some elementary notions of Pierce expansion and then define the error-sum function of Pierce expansion. In Section \ref{Section 3}, we investigate the basic properties, e.g., boundedness and continuity, of the error-sum function. In Section \ref{Section 4}, we determine the Hausdorff dimension, the box-counting dimension, and the covering dimension of the graph of the error-sum function.

Throughout the paper, $\mathbb{N}$ denotes the set of positive integers, $\mathbb{N}_0$ the set of non-negative integers, and $\mathbb{N}_\infty \coloneqq \mathbb{N} \cup \{ \infty \}$ the set of extended positive integers. Following the convention, we define $\infty + c \coloneqq \infty$ and $c/\infty \coloneqq 0$ for any constant $c \in \mathbb{R}$. We denote the Lebesgue measure on $[0,1]$ by $\lambda$. For any subset $A$ of a topological space $X$, the closure of $A$ is denoted by $\overline{A}$. Given any function $g \colon A \to B$, we write the preimage of any singleton $\{ b \} \subseteq B$ under $g$ simply as $g^{-1}(b)$ instead of $g^{-1}(\{ b \})$.

\section{Pierce expansion} \label{Section 2}

This section is devoted to introducing some basic notions of the Pierce expansion. We refer the reader to \cite{DDF22, ES91, Fan15, PVB98, Pie29, Rem54, Sha84, Sha86, Sha93, Var17, VBP99} or \cite[Chapter~2]{Sch95} for arithmetic and metric properties of the Pierce expansion.

The classical Pierce expansion is concerned with the numbers in the half-open unit interval $(0,1]$. In this paper, we extend our scope to the numbers in the closed unit interval $I \coloneqq [0,1]$. This extension is consistent with our use of $\mathbb{N}_\infty$ instead of $\mathbb{N}$ in this paper.

To dynamically generate the Pierce expansion of $x \in I$, we begin with two maps $d_1 \colon I \to \mathbb{N}_\infty$ and $T \colon I \to I$ given by
\[
d_1(x) \coloneqq \begin{cases} \left\lfloor 1/x \right\rfloor, &\text{if } x \neq 0; \\ \infty, &\text{if } x = 0, \end{cases} \quad \text{and} \quad T(x) \coloneqq \begin{cases} 1 - d_1 (x) \cdot x, &\text{if } x \neq 0; \\ 0, &\text{if } x =0, \end{cases}
\]
respectively, where $\lfloor y \rfloor$ denotes the largest integer not exceeding $y \in \mathbb{R}$. Observe that by definition, for each $n \in \mathbb{N}$, we have $d_1(x) =n$ if and only if $x \in I$ lies in the interval $( 1/(n+1), 1/n ]$ on which $T$ is linear. For each $n \in \mathbb{N}$, we write $T^n$ for the $n$th iterate of $T$, and $T^0 \coloneqq \id_{I}$. For notational convenience, we write $T^n x$ for $T^n (x)$ whenever no confusion could arise. 

Given $x \in I$, we define the sequence of {\em digits} $( d_n (x) )_{n \in \mathbb{N}}$ by $d_n(x) \coloneqq d_1 (T^{n-1}x)$ for each $n \in \mathbb{N}$. Then, for any $n \in \mathbb{N}$, by definitions of the map $T$ and the digits, we have
\begin{align} \label{T n-1 formula}
T^{n-1}x = \frac{1}{d_1(T^{n-1}x)} - \frac{T(T^{n-1}x)}{d_1(T^{n-1}x)} = \frac{1}{d_n(x)} - \frac{T^nx}{d_n(x)}.
\end{align}

We recall two well-known facts about the digits in the following proposition. In particular, part (i) characterizes the digit sequence, and it is stated in any study of Pierce expansions with or without proof. We include the proof to make it clear that the replacement of $(0,1]$ and $\mathbb{N}$ by $I$ and $\mathbb{N}_\infty$, respectively, does not violate the basic properties.

\begin{proposition} [See \cite{Sha86} and {\cite[Proposition~2.2]{Fan15}}] \label{basic facts}
Let $x \in I$ and $n \in \mathbb{N}$. Then the following hold.
\begin{enumerate}[label=(\roman*)]
\item
$d_{n+1}(x) \geq d_n(x) + 1$.
\item
$d_n(x) \geq n$.
\end{enumerate}
\end{proposition}

\begin{proof}
(i) 
If $d_n (x) = \infty$, then $T^{n-1}x=0$ and so $T^n x = 0$, which implies that $d_{n+1}(x) = \infty$. Since $\infty+1 = \infty$ by convention, we conclude that $d_{n+1}(x) = d_n(x)+1$.

Assume $d_n(x) \in \mathbb{N}$. Then $T^{n-1}x \neq 0$ and $T^{n-1} x \in ( 1/(k+1), 1/k ]$ for some $k \in \mathbb{N}$. Hence, $d_n(x) = k$, and by definition of $T$, it follows that $T^n x \in [ 0, 1/(k+1) )$. Now $d_{n+1}(x) = \infty$ if $T^nx=0$, and $k+1 \leq d_{n+1} (x) < \infty$ if $T^nx \neq 0$. In either case, $d_{n+1}(x) \geq d_n(x) + 1$.

(ii)
Clearly, $d_1(x) \geq 1$ by definition of $d_1$. If $n \geq 2$, using part (i) $(n-1)$ times, we obtain $d_n(x) \geq d_{n-1}(x) + 1 \geq \dotsb \geq d_1(x) + (n-1)$, and thus $d_n(x) \geq n$.
\end{proof}

We shall consider a symbolic space which is a subspace of $\mathbb{N}_\infty^\mathbb{N}$ closely related to Pierce expansions. Let $\Sigma_0 \coloneqq \{ (\sigma_k)_{k \in \mathbb{N}} \in \{ \infty \}^\mathbb{N} \}$, and for each $n \in \mathbb{N}$, let
\begin{align*}
\Sigma_n 
&\coloneqq \{ (\sigma_k)_{k \in \mathbb{N}} \in \mathbb{N}^n \times \{ \infty \}^{\mathbb{N} \setminus \{ 1, \dotsc, n \}} : \sigma_1 < \sigma_2 < \dotsb < \sigma_n  \}.
\end{align*}
For ease of notation, we will occasionally write $(\sigma_k)_{k \in \mathbb{N}} \in \Sigma_n$ as $(\sigma_1, \dotsc, \sigma_n)$ in place of $(\sigma_1, \dotsc, \sigma_n, \allowbreak \infty, \infty, \dotsc)$. We also define
\begin{align*}
\Sigma_\infty &\coloneqq \{ (\sigma_k)_{k \in \mathbb{N}} \in \mathbb{N}^\mathbb{N} : \sigma_k < \sigma_{k+1} \text{ for all } k \in \mathbb{N} \}.
\end{align*}
Then $\Sigma_n$, $n \in \mathbb{N}_0$, consists of a sequence with strictly increasing $n$ initial terms and $\infty$ for the remaining terms, and $\Sigma_\infty$ consists of strictly increasing infinite sequences of positive integers. Finally, let
\begin{align*}
\Sigma &\coloneqq \bigcup_{n \in \mathbb{N}_0} \Sigma_n \cup \Sigma_\infty
\end{align*}
in $\mathbb{N}_\infty^\mathbb{N}$. Each element of $\Sigma$ is said to be a {\em Pierce sequence}. In view of Proposition \ref{basic facts}(i), for any $x \in I$, the digit sequence $(d_n(x))_{n \in \mathbb{N}}$ is a Pierce sequence. We say $\sigma \coloneqq (\sigma_k)_{k \in \mathbb{N}} \in \Sigma$ is {\em realizable} if there exists $x \in I$ such that $d_k(x) = \sigma_k$ for all $k \in \mathbb{N}$, and we denote by $\Sigma_{\realizable}$ the collection of all realizable Pierce sequences. Note that, for any $(\sigma_n)_{n \in \mathbb{N}} \in \Sigma$, we have
\begin{align} \label{sigma n bound}
\sigma_n \geq n
\end{align}
for all $n \in \mathbb{N}$, which is analogous to Proposition \ref{basic facts}(ii).

It is well known that for each $x \in I$, the iterations of $T$ yield a unique expansion
\begin{align} \label{series expansion}
x = \sum_{n=1}^\infty \frac{(-1)^{n+1}}{d_1(x) \dotsm d_n(x)} = \cfrac{1-\cfrac{1-\cfrac{1-\dotsb}{d_3(x)}}{d_2(x)}}{d_1(x)},
\end{align}
where the digit sequence $( d_n(x) )_{n \in \mathbb{N}}$ is a realizable Pierce sequence. (See Proposition \ref{characterization of Sigma re} below.) The expression \eqref{series expansion} is called the {\em Pierce expansion}, {\em Pierce (ascending) continued fraction}, or {\em alternating Engel expansion} of $x$. We denote \eqref{series expansion} by
\[
x = [d_1(x), d_2(x), \dotsc, d_n(x), \dotsc]_P.
\]
For brevity, if the digit is $\infty$ at some point and on, i.e., if $x = [d_1(x), \dotsc, d_n(x), \infty, \infty, \dotsc]_P$ with $d_n(x) < \infty$ for some $n \in \mathbb{N}$, then we write $x = [d_1(x), \dotsc, d_n(x)]_P$ and say that $x$ has a {\em finite} Pierce expansion of length $n$. As mentioned in Section \ref{Section 1}, it is a classical result that the Pierce expansion of $x \in (0,1]$ is finite if and only if $x$ is rational. Since $0 =  [\infty, \infty, \dotsc]_P$, we may write the Pierce expansion of $0$ as $[ \, ]_P$, which is of length zero. Thus, $x \in I$ has a finite Pierce expansion if and only if $x$ is rational.

\begin{proposition} [See {\cite[pp. 23--24]{Sha86}}] \label{digit condition proposition}
For any $x \in I$, if its Pierce expansion is of length $n \geq 2$, then $d_{n-1}(x) + 1 < d_n(x)$.
\end{proposition}

\begin{proof}
The result follows from the definition of the digits. To see this, suppose otherwise. Put $M \coloneqq d_{n-1}(x)$ for some $M \in \mathbb{N}$, so that $d_{n}(x) = M+1$ by Proposition \ref{basic facts}(i). Since $d_{n+1}(x) = \infty$, we have $T^nx=0$. By \eqref{T n-1 formula}, we see that
\[
T^{n-2} x = \frac{1}{d_{n-1}(x)} - \frac{1}{d_{n-1}(x)} \left( \frac{1}{d_n(x)} - \frac{T^nx}{d_n(x)} \right) = \frac{1}{M} - \frac{1}{M(M+1)} = \frac{1}{M+1},
\]
and so $T^{n-1} x =0$. It follows that $d_n(x) = \infty$, which is a contradiction.
\end{proof}

We denote by $f \colon I \to \Sigma$ the map sending a number in $I$ to its sequence of Pierce expansion digits, that is, for each $x \in I$, $f$ is given by
\begin{align*} 
f(x) \coloneqq (d_n(x))_{n \in \mathbb{N}} = (d_1(x), d_2(x), d_3(x), \dotsc).
\end{align*}
Clearly, $f$ is well defined. We also note that $f(I) = \Sigma_{\realizable}$ by definition. 

Conversely, we shall introduce a function mapping a Pierce sequence to a real number in $I$ by means of the formula modelled on \eqref{series expansion}. Define a map $\varphi \colon \Sigma \to I$ by
\begin{align*} 
\varphi ( \sigma ) \coloneqq \sum_{n=1}^\infty \frac{(-1)^{n+1}}{\sigma_1 \dotsm \sigma_n} =  \frac{1}{\sigma_1} - \frac{1}{\sigma_1 \sigma_2} +  \dotsb + \frac{(-1)^{n+1}}{\sigma_1 \dotsm \sigma_n} + \dotsb
\end{align*}
for each $\sigma \coloneqq (\sigma_n)_{n \in \mathbb{N}} \in \Sigma$. Observe that $\varphi$ is well defined since $\sum_{n=1}^\infty 1/(\sigma_1 \dotsm \sigma_n) \leq \sum_{n=1}^\infty 1/n! < \infty$, where the first inequality follows from \eqref{sigma n bound}.

We rephrase \cite[Proposition~2.1]{Fan15} in terms of the maps $f$ and $\varphi$ in the following proposition. According to Fang \cite{Fan15}, the proposition is credited to Remez \cite{Rem54}. See also \cite[Section~4.1]{DDF22}.

Let $E \coloneqq I \cap \mathbb{Q}$ and $E' \coloneqq E \setminus \{ 0, 1 \} = (0,1) \cap \mathbb{Q}$.

\begin{proposition}[{\cite[Proposition~2.1]{Fan15}}] \label{inverse image of phi}
Let $x \in I$. Then the following hold.
\begin{enumerate}[label=(\roman*)]
\item
If $x \in E'$, then we have $\varphi^{-1}( x ) = \{ \sigma, \sigma' \}$, where
\begin{align*}
\sigma &\coloneqq (d_1(x), d_2(x), \dotsc, d_{n-1}(x), d_n(x)) = f(x) \in \Sigma_n \cap \Sigma_{\realizable}, \\
\sigma' &\coloneqq (d_1(x), d_2(x), \dotsc, d_{n-1}(x), d_n(x)-1, d_n(x)) \in \Sigma_{n+1} \cap (\Sigma \setminus \Sigma_{\realizable}),
\end{align*}
for some $n \in \mathbb{N}$. 
\item
If $x \in I \setminus E'$, then we have $\varphi^{-1}( x ) = \{ \sigma \}$, where $\sigma \coloneqq f(x)$. More precisely, $\sigma \in \Sigma_\infty$ if $x \in I \setminus E$; $\sigma = (\infty, \infty, \dotsc) \in \Sigma_0$ if $x = 0$; $\sigma = ( 1, \infty, \infty, \dotsc ) \in \Sigma_1$ if $x = 1$.
\end{enumerate}
\end{proposition}

The following proposition is a characterization of the set of realizable Pierce sequences.

\begin{proposition} \label{characterization of Sigma re}
Let $\sigma \coloneqq (\sigma_k)_{k \in \mathbb{N}} \in \Sigma$. Then $\sigma \not \in \Sigma_{\realizable}$ if and only if $\sigma \in \Sigma_n$ for some $n \geq 2$ with $\sigma_{n-1} + 1 = \sigma_n$.
\end{proposition}

\begin{proof}
The reverse implication follows from Proposition \ref{digit condition proposition}, since there does not exist $x \in I$ whose Pierce expansion is given by $[d_1(x), \dotsc, d_n(x)]_P$ with $d_{n-1}(x) = \sigma_{n-1}$ and $d_n(x) = \sigma_{n-1}+1$.

Now, for the forward implication, suppose $\sigma \in \Sigma \setminus \Sigma_{\realizable}$. Put $x \coloneqq \varphi (\sigma) \in I$. Then $\sigma \neq f(x)$ by definition of $\Sigma_{\realizable}$. The preimage $\varphi^{-1}(x)$ contains $\sigma$ and by Proposition \ref{inverse image of phi}, it is either a singleton or a doubleton. If $\varphi^{-1}(x) = \{ \sigma \}$, then by Proposition \ref{inverse image of phi}(ii), we have $\sigma = f(x)$, which is a contradiction. Hence, $\varphi^{-1}(x)$ is a doubleton and by Proposition \ref{inverse image of phi}(i), it follows that $\sigma \in \Sigma_{n}$ for some $n \geq 2$ with $\sigma_{n-1} = d_{n-1}(x)-1$ and $\sigma_n = d_{n-1}(x)$, in which case $\sigma_{n-1}+1 = \sigma_n$.
\end{proof}

For each $x \in I$ and $n \in \mathbb{N}$, define the $n$th {\em Pierce convergent} or {\em approximant}, $s_n \colon I \to \mathbb{R}$ by
\[
s_n (x) \coloneqq [d_1(x), \dotsc, d_n(x)]_P = \sum_{k=1}^n \frac{(-1)^{k+1}}{d_1 (x) \dotsm d_k(x)}.
\]
Then $s_n(x)$ is nothing but the $n$th partial sum of the Pierce expansion \eqref{series expansion}. Using \eqref{T n-1 formula} repeatedly, we find that
\begin{align} \label{simplified series}
x &= \frac{1}{d_1(x)} - \frac{Tx}{d_1(x)} \nonumber \\
&= \frac{1}{d_1(x)} - \frac{1}{d_1(x)} \left( \frac{1}{d_2(x)} - \frac{T^2x}{d_2(x)} \right) = \frac{1}{d_1(x)} - \frac{1}{d_1(x) d_2(x)} + \frac{T^2x}{d_1(x) d_2(x)} \nonumber \\
&= \dotsb \nonumber \\
&= \sum_{k=1}^n \frac{(-1)^{k+1}}{d_1(x) \dotsm d_k(x)} + \frac{(-1)^n T^nx}{d_1(x) \dotsm d_n(x)}
= s_n(x) + \frac{(-1)^n T^nx}{d_1(x) \dotsm d_n(x)}.
\end{align}

For every $x \in I$, we define 
\[
\mathcal{E} (x) \coloneqq \sum_{n=1}^\infty ( x - s_n (x) )
\]
and call $\mathcal{E} \colon I \to \mathbb{R}$ the {\em error-sum function of Pierce expansions} on $I$. 
Note that for any $x \in I$, by \eqref{simplified series}, Proposition \ref{basic facts}(ii), and boundedness of $T$, we have
\[
\left| x - s_n (x) \right| = \left| \frac{(-1)^n T^nx}{d_1(x) \dotsm d_n(x)} \right| \leq \frac{1}{n!} \to 0
\]
as $n \to \infty$, with $\sum_{n=1}^\infty 1/n! < \infty$. It follows that $\mathcal{E}(x)$ is well defined as an absolutely and uniformly convergent series (or as a series with finitely many non-zero terms if $T^nx=0$ for some $n \in \mathbb{N}$). 

Defining an error-sum function on $\Sigma$ is in order. For each $n \in \mathbb{N}$, define the $n$th partial sum $\varphi_n \colon \Sigma \to \mathbb{R}$ by
\[
\varphi_n(\sigma) \coloneqq \sum_{k=1}^n \frac{(-1)^{k+1}}{\sigma_1 \dotsm \sigma_k} = \frac{1}{\sigma_1} - \frac{1}{\sigma_1 \sigma_2} + \dotsb + \frac{(-1)^{n+1}}{\sigma_1 \sigma_2 \dotsm \sigma_n}
\]
for $\sigma \coloneqq (\sigma_k)_{k \in \mathbb{N}} \in \Sigma$.
For every $\sigma \in \Sigma$, we define
\begin{align} \label{definition of E star}
\mathcal{E}^* (\sigma) \coloneqq \sum_{n=1}^\infty ( \varphi (\sigma) - \varphi_n(\sigma) )
\end{align}
and call $\mathcal{E}^* \colon \Sigma \to \mathbb{R}$ the {\em error-sum function of Pierce sequences} on $\Sigma$. Notice that for any $\sigma \coloneqq (\sigma_n)_{n \in \mathbb{N}} \in \Sigma$, by \eqref{sigma n bound}, we have
\begin{align} \label{phi - phi n}
| \varphi (\sigma) - \varphi_n (\sigma) | 
&= \frac{1}{\sigma_1 \dotsm \sigma_{n}} \left| \frac{1}{\sigma_{n+1}} - \sum_{j=1}^\infty \left( \frac{1}{\sigma_{n+1} \dotsm \sigma_{n+2j}} - \frac{1}{\sigma_{n+1} \dotsm \sigma_{n+(2j+1)}} \right)  \right| \nonumber \\
&\leq \frac{1}{\sigma_1 \dotsm \sigma_{n+1}} \leq \frac{1}{(n+1)!} \to 0 
\end{align}
as $n \to \infty$, with $\sum_{n=1}^\infty 1/(n+1)! < \infty$. Hence, the series converges absolutely and uniformly on $\Sigma$, and it follows that $\mathcal{E}^*(\sigma)$ is well defined.

\section{Some basic properties of $\mathcal{E}(x)$} \label{Section 3}

This section is devoted to investigating some basic properties of the error-sum function of Pierce expansions $\mathcal{E} \colon I \to \mathbb{R}$. It will usually be done by the aid of the symbolic space $\Sigma$ and the error-sum function of Pierce sequences $\mathcal{E}^* \colon \Sigma \to \mathbb{R}$.

\subsection{Symbolic space $\Sigma$}

We equip $\mathbb{N}$ with the discrete topology and consider $(\mathbb{N}_\infty, \mathcal{T})$ as its one-point compactification, so that a subset in $(\mathbb{N}_\infty, \mathcal{T})$ is open if and only if it is either a subset of $\mathbb{N}$ or a set whose complement with respect to $\mathbb{N}_\infty$ is a finite set in $\mathbb{N}$. 

For a metric space $(X,d)$, we denote by $B_d(x;r)$ the $d$-open ball centered at $x \in X$ with radius $r>0$, i.e., $B_d(x;r) \coloneqq \{ y \in X : d(x,y) < r \}$.

\begin{lemma} \label{metric on N infinity}
Define $\rho : \mathbb{N}_\infty \times \mathbb{N}_\infty \to \mathbb{R}$ by 
\[
\rho (x, y) \coloneqq
\begin{dcases}
\frac{1}{x} + \frac{1}{y}, &\text{if } x \neq y; \\ 
0, &\text{if } x=y,
\end{dcases}
\]
for $x,y \in \mathbb{N}_\infty$. Then $\rho$ is a metric on $\mathbb{N}_\infty$ and induces $\mathcal{T}$.
\end{lemma}

\begin{proof}
It is straightforward to check that $\rho$ is a metric on $\mathbb{N}_\infty$, so we prove the second assertion only.

Let $O \in \mathcal{T}$. We show that $O$ is $\rho$-open. Suppose $x \in O$. Then either $x \in \mathbb{N}$ or $x = \infty$. If $x \in \mathbb{N}$, then $x \in \{ x \} = B_{\rho} ( x ; 1/(2x) ) \subseteq O$, and hence $O$ is a neighborhood of $x$ in the $\rho$-topology. Now assume $x = \infty$. Then $\mathbb{N}_\infty \setminus O \subseteq \mathbb{N}$ is finite. So we can find a $K \in \mathbb{N}$ such that every $n \geq K$ is in $O$. Then $x \in \{ \infty \} \cup \{ K+1, K+2, K+3, \dotsc \} = B_\rho ( \infty; 1/K ) \subseteq O$. Thus, $O$ is a neighborhood of $x$ in the $\rho$-topology. Since $x \in O$ is arbitrary, we conclude that $O$ is $\rho$-open.

Conversely, let $U \subseteq \mathbb{N}_\infty$ be a $\rho$-open set. Suppose $x \in U$. Then either $x \in \mathbb{N}$ or $x = \infty$. Assume first $x \in \mathbb{N}$. But then $\{ x \}$, which is open in $(\mathbb{N}_\infty, \mathcal{T})$, satisfies $x \in \{ x \} \subseteq U$. Hence, $U$ is a neighborhood of $x$ in $(\mathbb{N}_\infty, \mathcal{T})$. Now assume $x = \infty$. Since $U$ is $\rho$-open, we can find an $r>0$ such that $B_{\rho} (\infty; r) \subseteq U$. Note that $y \in B_{\rho} (\infty; r)$ if and only if $y = \infty$ or $y > 1/r$. Then $B_\rho (\infty; r)$ is equal to $\{ \infty \} \cup \{ \lfloor 1/r \rfloor + 1, \lfloor 1/r \rfloor + 2, \lfloor 1/r \rfloor + 3, \dotsc \}$ which is an open set in $(\mathbb{N}_\infty, \mathcal{T})$ containing $x = \infty$. Thus $U$ is a neighborhood of $x$ in $(\mathbb{N}_\infty, \mathcal{T})$. Since $x \in U$ is arbitrary, it follows that $U$ is open in $(\mathbb{N}_\infty, \mathcal{T})$.
\end{proof}

Tychonoff's theorem tells us that $\mathbb{N}_\infty^\mathbb{N}$ is compact in the product topology, as a (countable) product of a compact space $(\mathbb{N}_\infty, \mathcal{T})$. It is easy to see that any non-empty open set in the product topology contains a non-Pierce sequence, so that $\Sigma$ is not open in $\mathbb{N}_\infty^\mathbb{N}$. However, $\Sigma$ is closed in the product topology.

\begin{lemma} \label{Sigma is closed}
The subspace $\Sigma$ is closed in the product space $\mathbb{N}_\infty^\mathbb{N}$, and so $\Sigma$ is compact in the product topology.
\end{lemma}

\begin{proof}
We show that $\mathbb{N}_\infty^\mathbb{N} \setminus \Sigma$ is open, i.e., every point of $\mathbb{N}_\infty^\mathbb{N} \setminus \Sigma$ is an interior point. Let $\sigma \coloneqq (\sigma_n)_{n \in \mathbb{N}} \in \mathbb{N}_\infty^\mathbb{N} \setminus \Sigma$. By definition of $\mathbb{N}_\infty^\mathbb{N} \setminus \Sigma$, we can find an index $K \in \mathbb{N}$ such that $\sigma_K \geq \sigma_{K+1} \neq \infty$, say $M \coloneqq \sigma_{K+1} \in \mathbb{N}$. Consider a subset $O \subseteq \mathbb{N}_\infty^\mathbb{N}$ given by
\[
O \coloneqq\mathbb{N}_\infty^{\{ 1, \dotsc, K-1 \}} \times (\mathbb{N}_\infty \setminus \{ 1, 2, \dotsc, M-1 \} ) \times \{ M \} \times \mathbb{N}_\infty^{\mathbb{N} \setminus \{ 1, 2, \dotsc, K+1 \}}.
\]
Then $O$ is open in $\mathbb{N}_\infty^\mathbb{N}$ by definition of the product topology, since $\mathbb{N}_\infty \setminus \{ 1, 2, \dotsc, M-1 \}$ and $\{ M \}$ are open in $(\mathbb{N}_\infty, \mathcal{T})$. It is clear that $\sigma \in O \subseteq \mathbb{N}_\infty^\mathbb{N} \setminus \Sigma$. Thus $\sigma$ is an interior point of $\mathbb{N}_\infty^\mathbb{N} \setminus \Sigma$, and this proves that $\Sigma$ is closed in $\mathbb{N}_\infty^\mathbb{N}$.

The second assertion follows immediately since $\mathbb{N}_\infty^\mathbb{N}$ is compact in the product topology by Tychonoff's theorem, so that its closed subspace $\Sigma$ is compact.
\end{proof}

Let $\sigma \coloneqq (\sigma_n)_{n \in \mathbb{N}}$ and $\tau \coloneqq (\tau_n)_{n \in \mathbb{N}}$ be any two elements in $\mathbb{N}_\infty^\mathbb{N}$. We define $\rho^\mathbb{N} \colon \mathbb{N}_\infty^\mathbb{N} \times \mathbb{N}_\infty^\mathbb{N} \to \mathbb{R}$ by
\[
\rho^\mathbb{N} (\sigma, \tau) \coloneqq 
\sum_{n=1}^\infty \frac{\rho (\sigma_n, \tau_n)}{n!},
\]
where $\rho \colon \mathbb{N}_\infty \times \mathbb{N}_\infty \to \mathbb{R}$ is defined as in Lemma \ref{metric on N infinity}. Notice that we have $\rho (\sigma_n, \tau_n) \leq 1/\sigma_n + 1/\tau_n \leq 1/n + 1/n \leq 2$ for each $n \in \mathbb{N}$ by \eqref{sigma n bound}, and $\sum_{n=1}^\infty 2 / n! < \infty$. We deduce that $\rho^\mathbb{N}$ is well defined.

\begin{lemma} \label{metric on Sigma}
The function $\rho^\mathbb{N}$ is a metric on $\mathbb{N}_\infty^\mathbb{N}$ and the topology induced by $\rho^\mathbb{N}$ is equivalent to the product topology on $\mathbb{N}_\infty^\mathbb{N}$.
\end{lemma}

\begin{proof}
It is straightforward to check that $(\mathbb{N}_\infty^\mathbb{N},\rho^\mathbb{N})$ is a metric space. The proof of the second assertion is almost identical to the standard proof of the well-known fact that any countable product of metric spaces is metrizable. So we omit the details.
\end{proof}

\begin{lemma} \label{Sigma is compact}
The metric space $(\Sigma, \rho^\mathbb{N})$ is compact.
\end{lemma}

\begin{proof}
The lemma is immediate from Lemmas \ref{Sigma is closed} and \ref{metric on Sigma}.
\end{proof}

For a given $\sigma \coloneqq (\sigma_k)_{k \in \mathbb{N}} \in \Sigma$, we define $\sigma^{(n)} \coloneqq (\tau_k)_{k \in \mathbb{N}}\in \Sigma$ for each $n \in \mathbb{N}$, by 
\[
\tau_k
\coloneqq 
\begin{cases}
\sigma_k, &\text{if } 1 \leq k \leq n; \\
\infty, &\text{otherwise},
\end{cases}
\]
i.e., $\sigma^{(n)} = (\sigma_1, \dotsc, \sigma_n, \infty, \infty, \dots)$. It is worth pointing out that it is not always the case that $\sigma^{(n)} \in \Sigma_n$, since we might have $\sigma_k = \infty$ for some $1 \leq k \leq n$. 

Fix $n \in \mathbb{N}$ and $\sigma \in \Sigma_n$. Let $\Upsilon_\sigma$ be the collection of sequences in $\Sigma$ defined as
\[
\Upsilon_\sigma \coloneqq \{ \upsilon \in \Sigma : \upsilon^{(n)} = \sigma \},
\]
and we call $\Upsilon_\sigma$ the {\em cylinder set} of order $n$ associated with $\sigma$. Then $\Upsilon_\sigma$ consists of all sequences in $\Sigma$ whose initial $n$ terms agree with those of $\sigma$. By Lemma \ref{Sigma is closed}, it is clear that $\Upsilon_\sigma$ is compact in $\Sigma$ as a closed set in a compact space. Since $\Upsilon_\sigma$ is open in $\Sigma$ as well, it follows that $\Sigma \setminus \Upsilon_\sigma$ is compact by the same lemma. We also define the {\em fundamental interval} of order $n$ associated with $\sigma \coloneqq (\sigma_k)_{k \in \mathbb{N}}$ by
\begin{align*} 
I_\sigma \coloneqq \{ x \in I : d_k(x) = \sigma_k \text{ for all } 1 \leq k \leq n \} = f^{-1} (\Upsilon_\sigma).
\end{align*}
Then any number $x \in I_\sigma$ has its Pierce expansion beginning with $(\sigma_k)_{k=1}^n$, i.e.,
\begin{align*}
x 
&= [d_1(x), \dotsc, d_n(x), d_{n+1}(x), d_{n+2}(x), \dots]_P \\
&= [\sigma_1, \dotsc, \sigma_n, d_{n+1}(x), d_{n+2}(x), \dotsc]_P.
\end{align*}
In view of the following proposition, the reason for $I_\sigma$ being called an interval should be clear.

For each $n \in \mathbb{N}$ and $\sigma \coloneqq (\sigma_k)_{k \in \mathbb{N}} \in \Sigma_n$, we write $\widehat{\sigma} \coloneqq (\widehat{\sigma}_k)_{k \in \mathbb{N}} \in \Sigma_n$, where the $\widehat{\sigma}_k$ are given by
\[
\widehat{\sigma}_k =
\begin{cases}
\sigma_n+1, &\text{if } k = n; \\
\sigma_k, &\text{otherwise},
\end{cases}
\]
i.e., $\widehat{\sigma} = (\sigma_1, \dotsc, \sigma_{n-1}, \sigma_n+1, \infty, \infty, \dotsc)$. 

\begin{proposition} [{\cite[Theorem~1]{Sha86}}] \label{I sigma}
Let $n \in \mathbb{N}$ and $\sigma \coloneqq (\sigma_k)_{k \in \mathbb{N}} \in \Sigma_n$. If $\sigma \in \Sigma_{\realizable}$, then
\begin{align} \label{fundamental interval 1}
I_\sigma = \begin{cases} ( \varphi (\widehat{\sigma}), \varphi (\sigma)], &\text{if $n$ is odd}; \\ [ \varphi (\sigma), \varphi (\widehat{\sigma})), &\text{if $n$ is even}. \end{cases}
\end{align}
If $\sigma \not \in \Sigma_{\realizable}$, we have instead that $I_\sigma$ is an open interval with the same endpoints, i.e., 
\begin{align*} \label{fundamental interval 2}
I_\sigma = \begin{cases} ( \varphi (\widehat{\sigma}), \varphi (\sigma)), &\text{if $n$ is odd}; \\ ( \varphi (\sigma), \varphi (\widehat{\sigma})), &\text{if $n$ is even}. \end{cases} \tag{\ref{fundamental interval 1}$'$}
\end{align*}
Consequently, the length of $I_\sigma$ is
\begin{align} \label{length of I sigma}
\lambda (I_\sigma) = | \varphi (\sigma) - \varphi (\widehat{\sigma})| = \frac{1}{\sigma_1 \dotsm \sigma_{n-1} \sigma_n (\sigma_n+1)}.
\end{align}
\end{proposition}

We illustrate the exclusion of the endpoint $\varphi (\sigma)$ in \eqref{fundamental interval 2} by an example. Consider two sequences $\sigma \coloneqq (2) \in \Sigma_1 \cap \Sigma_{\realizable}$ and $\sigma' \coloneqq (1,2) \in \Sigma_2 \cap (\Sigma \setminus \Sigma_{\realizable})$. Then $\varphi (\sigma) = 1/2$ and $\varphi (\sigma') = 1/1 - 1/(1 \cdot 2) = 1/2$ are equal, and so they have the same Pierce expansion, namely $[2]_P$. It follows by the definition of fundamental intervals that $I_{\sigma}$ contains $\varphi (\sigma)$, whereas $I_{\sigma'}$ fails to contain $\varphi (\sigma')$.

For later use, we record an upper bound for $\lambda (I_\sigma)$ derived from \eqref{length of I sigma}. For each $\sigma \coloneqq (\sigma_k)_{k \in \mathbb{N}} \in \Sigma_n$, since $\sigma_k \geq k$ for $1 \leq k \leq n$ by \eqref{sigma n bound}, we have that 
\begin{align} \label{bound for length of I sigma}
\lambda (I_\sigma) \leq \frac{1}{(1)(2) \dotsm (n-1)(n)(n+1)} = \frac{1}{(n+1)!}.
\end{align}

\subsection{Mappings $\varphi \colon \Sigma \to I$ and $f \colon I \to \Sigma$}

By definition, the following observation is immediate. 

\begin{observation}
We have $\varphi \circ f = \id_{I}$ and $f \circ (\varphi|_{\Sigma_{\realizable}}) = \id_{\Sigma_{\realizable}}$, where $\varphi|_{\Sigma_{\realizable}}$ is the restriction of $\varphi$ to $\Sigma_{\realizable}$, but $f \circ \varphi \neq \id_\Sigma$ in general.
\end{observation}

For a fixed $\sigma \in \Sigma_n$ for some $n \in \mathbb{N}$, we can explicitly describe the relation between the cylinder set $\Upsilon_\sigma \subseteq \Sigma$ and the fundamental interval $I_\sigma \subseteq I$ in terms of the map $f \colon I \to \Sigma$. We first observe the following from the definition $I_\sigma = f^{-1}(\Upsilon_\sigma)$.

\begin{observation}
Let $n \in \mathbb{N}$ and $\sigma \in \Sigma_n$. Then $f(I_\sigma) \subseteq \Upsilon_\sigma$.
\end{observation}

The inclusion in the above observation is proper, i.e., $f(I_\sigma) \subsetneq \Upsilon_\sigma$, and by Proposition \ref{characterization of Sigma re} we explicitly have
\begin{align} \label{cylinder set minus f image}
\Upsilon_\sigma \setminus f(I_\sigma) = \bigcup_{m \geq n} \{ \tau \coloneqq (\tau_k)_{k \in \mathbb{N}}  \in \Sigma_{m} : \tau^{(n)} = \sigma \text{ and } \tau_{m-1}+1 = \tau_{m} \} \subseteq \Sigma \setminus f(I).
\end{align}
But $f(I_\sigma)$ is not significantly smaller than $\Upsilon_\sigma$ in the sense that $f (I_\sigma)$ is dense in $\Upsilon_\sigma$.

\begin{lemma} \label{f image of fundamental interval is dense}
Let $n \in \mathbb{N}$ and $\sigma \in \Sigma_n$. Then $\overline{f (I_\sigma)} = \Upsilon_\sigma$. 
\end{lemma}

\begin{proof}
Put $\sigma \coloneqq (\sigma_k)_{k \in \mathbb{N}} \in \Sigma_n$. It is clear that $\Upsilon_\sigma$ is closed in $\Sigma$. Since $f(I_\sigma) \subseteq \Upsilon_\sigma$ and $\Upsilon_\sigma$ is closed, it suffices to show that any point in $\Upsilon_\sigma \setminus f(I_\sigma)$ is a limit point of $f(I_\sigma)$. Let $\upsilon \coloneqq (\upsilon_k)_{k \in \mathbb{N}} \in \Upsilon_\sigma \setminus f (I_\sigma)$. Then, by \eqref{cylinder set minus f image}, $\upsilon$ is of the form
\[
(\sigma_1, \dotsc, \sigma_{n-1}, \sigma_n, \dotsc, \upsilon_{m-1}, \upsilon_{m}, \infty, \infty, \dotsc),
\] 
where $\upsilon_{m-1} + 1 = \upsilon_m$, for some $m \geq n$. Consider a sequence $(\bm{\tau}_k)_{k \in \mathbb{N}}$ in $\Sigma$ given by
\[
\bm{\tau}_k \coloneqq (\sigma_1, \dotsc, \sigma_n, \dotsc, \upsilon_{m-1}, \upsilon_m, \upsilon_m + (k+1), \infty, \infty, \dotsc)
\]
for each $k \in \mathbb{N}$. Then $\bm{\tau}_k \in f(I_\sigma)$ for all $k \in \mathbb{N}$ by Proposition \ref{characterization of Sigma re}, since $\bm{\tau}_k \in \Sigma_{m+1}$ and $\upsilon_m + 1 < \upsilon_m + (k+1)$. Clearly, $\bm{\tau}_k \to \upsilon$ as $k \to \infty$. This completes the proof.
\end{proof}

Similarly, any sequence in $\Sigma$ can be approximated arbitrarily close by sequences in $f(I)$.

\begin{lemma} \label{closure of f I}
We have $\Sigma = \overline{f(I)}$.
\end{lemma}

\begin{proof}
Since $f(I) \subseteq \Sigma$, it suffices to show that any point in $\Sigma \setminus f(I)$ is a limit point of $f(I)$. Let $\sigma \coloneqq (\sigma_k)_{k \in \mathbb{N}} \in \Sigma \setminus f(I)$. Then $\sigma \in \Sigma \setminus \Sigma_{\realizable}$, so by Proposition \ref{characterization of Sigma re}, $\sigma \in \Sigma_n$ for some $n \geq 2$ with $\sigma_{n-1}+1 = \sigma_n$. Now, an argument similar to the one in the proof of Lemma \ref{f image of fundamental interval is dense} shows that there is a sequence in $f(I)$ converging to $\sigma$. Hence the result.
\end{proof}

We are now concerned with the continuity of two maps of interest. We first show that $\varphi \colon \Sigma \to I$ is a Lipschitz mapping.

\begin{lemma} \label{phi is Lipschitz} 
For any $\sigma, \tau \in \Sigma$, we have $|\varphi (\sigma) - \varphi (\tau)| \leq \rho^\mathbb{N} (\sigma, \tau)$.
\end{lemma}

\begin{proof}
Let $\sigma \coloneqq (\sigma_n)_{n \in \mathbb{N}}, \tau \coloneqq (\tau_n)_{n \in \mathbb{N}} \in \Sigma$. If $\sigma = \tau$, there is nothing to prove, so we suppose that $\sigma$ and $\tau$ are distinct. If $\sigma_1 \neq \tau_1$, then
\[
|\varphi (\sigma) - \varphi (\tau)| \leq |\varphi (\sigma)| + |\varphi (\tau)| \leq \frac{1}{\sigma_1} + \frac{1}{\tau_1} \leq \rho^{\mathbb{N}} (\sigma, \tau).
\]
Assume that $\sigma$ and $\tau$ share the initial block of length $n \in \mathbb{N}$, i.e., $\sigma^{(n)} = \tau^{(n)}$ and $\sigma_{n+1} \neq \tau_{n+1}$. Then
\begingroup
\allowdisplaybreaks
\begin{align*}
| \varphi (\sigma) - \varphi (\tau)| 
&= | ( \varphi (\sigma) - \varphi(\sigma^{(n)}) ) -  (\varphi (\tau) - \varphi (\tau^{(n)}) ) | \\
&= |(\varphi (\sigma) - \varphi_n (\sigma)) - (\varphi (\tau) - \varphi_n (\tau))| 
\leq |\varphi (\sigma) - \varphi_n (\sigma)| + |\varphi (\tau) - \varphi_n (\tau)| \\
&\leq \frac{1}{\sigma_1 \dotsm \sigma_{n}} \left( \frac{1}{\sigma_{n+1}} + \frac{1}{\tau_{n+1}} \right) 
\leq \frac{1}{n!} \left( \frac{1}{\sigma_{n+1}} + \frac{1}{\tau_{n+1}} \right) \\
&\leq \frac{1}{n!} \left( \frac{1}{\sigma_{n}} + \frac{1}{\tau_{n}} \right) 
\leq \rho^{\mathbb{N}} (\sigma, \tau),
\end{align*}
\endgroup
where we used \eqref{phi - phi n} and \eqref{sigma n bound} for the second and third inequalities, respectively.
\end{proof}

Now we prove that $f \colon I \to \Sigma$ is continuous at every irrational number and at two rational numbers $0$ and $1$.

\begin{lemma} \label{f is continuous at irrational}
The mapping $f \colon I \to \Sigma$ is continuous at every $x \in I \setminus E'$.
\end{lemma}

\begin{proof}
By Proposition \ref{inverse image of phi}(ii), it suffices to show that $f$ is continuous at $x \in I$ for which $\varphi^{-1} ( x )$ is a singleton. Suppose otherwise. Put $\{ \sigma \} \coloneqq \varphi^{-1} ( x )$ for some $\sigma \in \Sigma$. Then $f(x) = \sigma$ by Proposition \ref{inverse image of phi}(ii). Since $f$ is not continuous at $x$, we can find an $\varepsilon > 0$ and a sequence $(x_n)_{n \in \mathbb{N}}$ in $I$ such that $|x-x_n| < 1/n$ but $\rho^\mathbb{N} (\sigma, f(x_n)) \geq \varepsilon$ for all $n \in \mathbb{N}$. Since $(\tau_n)_{n \in \mathbb{N}} \coloneqq (f(x_n))_{n \in \mathbb{N}}$ is a sequence in a compact metric space $\Sigma$ (Lemma \ref{Sigma is compact}), there is a subsequence $(\tau_{n_k})_{k \in \mathbb{N}}$ converging to some $\tau \in \Sigma$. Note that $x_{n_k} = \varphi (f(x_{n_k})) = \varphi (\tau_{n_k})$ for each $k \in \mathbb{N}$. Now, by continuity of $\varphi$ (Lemma \ref{phi is Lipschitz}), we see that $x_{n_k} \to \varphi (\tau)$ as $k \to \infty$. Since $x$ is the limit of $(x_n)_{n \in \mathbb{N}}$, it follows that $x = \varphi (\tau)$. Thus $\tau = \sigma$ by the singleton assumption. But then $\rho^\mathbb{N} (\tau, f(x_{n_k})) = \rho^\mathbb{N} (\tau, \tau_{n_k}) \geq \varepsilon$ for all $k \in \mathbb{N}$, by our choice of $\varepsilon$ and $(x_n)_{n \in \mathbb{N}}$. This contradicts the convergence of $( \tau_{n_k})_{k \in \mathbb{N}}$ to $\tau$. Therefore, $f$ is continuous at $x \in I$ for which $\varphi^{-1} ( x )$ is a singleton, and hence at every $x \in I \setminus E'$.
\end{proof}

However, the continuity does not hold at any rational number in the open unit interval $(0,1)$. Notice in Proposition \ref{inverse image of phi}(i) that $\sigma \not \in \Upsilon_{\sigma'}$ and $\sigma' \not \in \Upsilon_\sigma$.

\begin{lemma} \label{f is discontinuous at rational}
Let $x \in E'$ and put $\varphi^{-1} ( x ) = \{ \sigma, \tau \}$. Then $f$ is not continuous at $x$, in particular, we have
\[
\lim_{\substack{t \to x \\ t \in I_\sigma}} f(t) = \sigma
\quad \text{and} \quad
\lim_{\substack{t \to x \\ t \not \in I_\sigma}} f(t) = \tau.
\]
\end{lemma}

\begin{proof}
The argument is similar to the proof of Lemma \ref{f is continuous at irrational}. The main difference in this proof is the use of compactness of $\Upsilon_\sigma$ and $\Sigma \setminus \Upsilon_\sigma$.

Suppose to the contrary that the first convergence fails to hold. We can find an $\varepsilon > 0$ and a sequence $(x_n)_{n \in \mathbb{N}}$ in $I_\sigma$ such that $|x-x_n| < 1/n$ but $\rho^\mathbb{N} (\sigma, f(x_n)) \geq \varepsilon$ for all $n \in \mathbb{N}$. Since $(\upsilon_n)_{n \in \mathbb{N}} \coloneqq (f(x_n))_{n \in \mathbb{N}}$ is a sequence in $f(I_\sigma) \subseteq \Upsilon_\sigma$, and $\Upsilon_\sigma$ is a compact metric space, there is a subsequence $(\upsilon_{n_k})_{k \in \mathbb{N}}$ converging to some $\upsilon \in \Upsilon_\sigma$. Note that $x_{n_k} = \varphi (f(x_{n_k})) = \varphi (\upsilon_{n_k})$ for each $k \in \mathbb{N}$. Now, by continuity of $\varphi$ (Lemma \ref{phi is Lipschitz}), we see that $x_{n_k} \to \varphi (\upsilon)$ as $k \to \infty$. Since $x$ is the limit of $(x_n)_{n \in \mathbb{N}}$, it follows that $x = \varphi (\upsilon)$. Thus $\upsilon = \sigma$ or $\upsilon = \tau$ by the doubleton assumption. Since $\tau \not \in \Upsilon_\sigma$ by Proposition \ref{inverse image of phi}(i), it must be that $\upsilon = \sigma$. But then $\rho^\mathbb{N} (\upsilon, f(x_{n_k})) = \rho^\mathbb{N} (\upsilon, \upsilon_{n_k}) \geq \varepsilon$ for all $k \in \mathbb{N}$, by our choice of $\varepsilon$ and $(x_n)_{n \in \mathbb{N}}$. This contradicts the convergence of $( \upsilon_{n_k})_{k \in \mathbb{N}}$ to $\upsilon$. Therefore, $\lim_{\substack{t \to x \\ t \in I_\sigma}} f(t) = \sigma$.

The proof for the second convergence is similar. First note that since $\Upsilon_\sigma \setminus f(I_\sigma)$ and $f(I)$ are disjoint by \eqref{cylinder set minus f image}, we have
\[
f(I \setminus I_\sigma) = f(I) \setminus f(I_\sigma) = [f(I) \cup (\Upsilon_\sigma \setminus f(I_\sigma))] \setminus [f(I_\sigma) \cup (\Upsilon_\sigma \setminus f(I_\sigma))] \subseteq \Sigma \setminus \Upsilon_\sigma,
\]
where the first equality follows from the injectivity of $f$. Now, in the preceding paragraph, by replacing $I_\sigma$ and $\Upsilon_\sigma$ by $I \setminus I_\sigma$ and $\Sigma \setminus \Upsilon_\sigma$, respectively, and exchanging the roles of $\sigma$ and $\tau$, we obtain the desired result.
\end{proof}

Notice that in the preceding lemma there is no additional assumption for $\sigma$ and $\tau$. Compare this with Proposition \ref{inverse image of phi}(i). Hence, Lemma \ref{f is discontinuous at rational} holds for either the case where $\sigma \in \Sigma_{\realizable}$ with $\tau \in \Sigma \setminus \Sigma_{\realizable}$ or where $\sigma \in \Sigma \setminus \Sigma_{\realizable}$ with $\tau \in \Sigma_{\realizable}$.

\subsection{The error-sum functions $\mathcal{E}$ and $\mathcal{E}^*$}

Now we first establish the relation between two error-sum functions $\mathcal{E} \colon I \to \mathbb{R}$ and $\mathcal{E}^* \colon \Sigma \to \mathbb{R}$.

\begin{lemma} \label{commutative diagram}
We have $\mathcal{E} = \mathcal{E}^* \circ f$, i.e., the following diagram commutes:
\begin{center}
\begin{tikzcd}[column sep=small]
I \arrow{rr}{\mathcal{E}} \arrow[swap, shift right=.75ex]{dr}{f}& &\mathbb{R} \\
& \Sigma \arrow[swap]{ur}{\mathcal{E}^*}  
\end{tikzcd}    
\end{center}
\end{lemma}

\begin{proof}
Let $x \in I$. Under the mapping $f \colon I \to \Sigma$, we obtain the sequence of Pierce expansion digits, namely, if $x = [d_1(x), d_2(x), \dotsc]_P$, then
we have that $\sigma \coloneqq f(x) = ( d_1(x), d_2(x), \dotsc ) \in \Sigma$. Now, by definition of $\varphi$, we have
\[
\varphi (\sigma) = \sum_{k=1}^\infty \frac{(-1)^{k+1}}{d_1(x) \dotsm d_k(x)} = x,
\]
and by definitions of $\varphi_n$ and $s_n$, we have, for each $n \in \mathbb{N}$, 
\[
\varphi_n(\sigma) = \sum_{k=1}^n \frac{(-1)^{k+1}}{d_1(x) \dotsm d_k(x)} = s_n(x).
\]
Thus $(\mathcal{E}^* \circ f)(x) = \mathcal{E}(x)$ for all $x \in I$.
\end{proof}

However, the equality $\mathcal{E} \circ \varphi = \mathcal{E}^*$ does not hold in general. For instance, consider $(2,3) \in \Sigma_2 \cap (\Sigma \setminus \Sigma_{\realizable})$. On one hand, $\varphi ((2,3)) = 1/2 - 1/(2 \cdot 3) = 1/3 = [3]_P$, and so $(\mathcal{E} \circ \varphi)((2,3)) = \mathcal{E}([3]_P) = 1/3 - 1/3 = 0$. On the other hand, $\mathcal{E}^*((2,3)) = ( 1/3 - 1/2 ) + ( 1/3 - ( 1/2 - 1/(2 \cdot 3) ) ) = - 1/6$. 

\begin{lemma} \label{sequence lemma}
For any $\sigma \coloneqq (\sigma_n)_{n \in \mathbb{N}} \in \Sigma$, the sequence $( n /(\sigma_1 \dotsm \sigma_{n+1}) )_{n \in \mathbb{N}}$ is monotonically decreasing to $0$ and the series $\sum_{n=1}^\infty n / (\sigma_1 \dotsm \sigma_{n+1})$ is convergent.
\end{lemma}

\begin{proof}
Note that $n > (n+1)/(n+2) \geq (n+1)/\sigma_{n+2}$ for each $n \in \mathbb{N}$ by \eqref{sigma n bound}. Then, again by \eqref{sigma n bound}, we have
\begin{align*}
\frac{n}{(n+1)!} \geq \frac{n}{\sigma_1 \dotsm \sigma_{n+1}} \geq \frac{n+1}{\sigma_1 \dotsm \sigma_{n+1} \sigma_{n+2}} \geq 0
\end{align*}
for every $n \in \mathbb{N}$, with $\sum_{n=1}^\infty n/(n+1)! < \infty$. Hence the result.
\end{proof}

We derive one simple formula for $\mathcal{E}^* \colon \Sigma \to \mathbb{R}$ which will be used frequently in the subsequent discussion.

\begin{lemma} \label{E star formula lemma}
Let $\sigma \coloneqq (\sigma_n)_{n \in \mathbb{N}} \in \Sigma$. Then
\begin{align} \label{E star formula}
\mathcal{E}^*(\sigma) = \sum_{n=1}^\infty \frac{(-1)^n n}{\sigma_1 \dotsm \sigma_{n+1}}.
\end{align}
\end{lemma}

\begin{proof}
Write
\begin{align*}
\mathcal{E}^* (\sigma)
&= \sum_{j=1}^\infty ( \varphi (\sigma) - \varphi_j (\sigma) )  
= \sum_{j=1}^\infty \bigg( \sum_{k=1}^\infty \frac{(-1)^{k+1}}{\sigma_1 \dotsm \sigma_k} - \sum_{k=1}^j \frac{(-1)^{k+1}}{\sigma_1 \dotsm \sigma_k}  \bigg)  \\
&= \sum_{j=1}^\infty \sum_{k=j+1}^\infty \frac{(-1)^{k+1}}{\sigma_1 \dotsm \sigma_k}.
\end{align*}
Notice that 
\[
\sum_{j=1}^\infty \sum_{k=j+1}^\infty \bigg| \frac{(-1)^{k+1}}{\sigma_1 \dotsm \sigma_k} \bigg|
= \sum_{n=1}^\infty \frac{n}{\sigma_1 \dotsm \sigma_{n+1}} < \infty
\]
by Lemma \ref{sequence lemma}. Thus, we may change the order of the double series to obtain
\begin{align*}
\mathcal{E}^* (\sigma)
&= \sum_{k=2}^\infty \sum_{j=1}^{k-1} \frac{(-1)^{k+1}}{\sigma_1 \dotsm \sigma_k} 
= \sum_{k=2}^\infty \frac{(-1)^{k+1}(k-1)}{\sigma_1 \dotsm \sigma_k} 
= \sum_{n=1}^\infty \frac{(-1)^n n}{\sigma_1 \dotsm \sigma_{n+1}},
\end{align*}
as desired.
\end{proof}

The boundedness of $\mathcal{E} \colon I \to \mathbb{R}$ readily follows.

\begin{theorem} \label{boundedness theorem}
For any $\sigma \in \Sigma$, we have $- 1/2 \leq \mathcal{E}^* (\sigma) \leq 0$. Consequently, $- 1/2 < \mathcal{E}(x) \leq 0$ for all $x \in I$.
\end{theorem}

\begin{proof}
We make use of \eqref{E star formula} and Lemma \ref{sequence lemma} to obtain both the desired upper and lower bounds. On one hand, for any $\sigma \coloneqq (\sigma_n)_{n \in \mathbb{N}} \in \Sigma$, we have
\begin{align*}
\mathcal{E}^*(\sigma) 
&= -\frac{1}{\sigma_1 \sigma_2} + \sum_{j=1}^\infty \left( \frac{2j}{\sigma_1 \dotsm \sigma_{2j+1}} - \frac{2j+1}{\sigma_1 \dotsm \sigma_{2j+2}} \right) 
\geq - \frac{1}{\sigma_1 \sigma_2} \geq - \frac{1}{1 \cdot 2} = - \frac{1}{2},
\end{align*}
where the last inequality follows from \eqref{sigma n bound}. Notice that the equalities hold if and only if $\sigma = (1,2) \in \Sigma_2 \cap (\Sigma \setminus \Sigma_{\realizable})$. On the other hand, for any $\sigma \coloneqq (\sigma_n)_{n \in \mathbb{N}} \in \Sigma$, we have
\begin{align*}
\mathcal{E}^*(\sigma) 
&= - \sum_{j=1}^\infty \left( \frac{2j-1}{\sigma_1 \dotsm \sigma_{2j}} - \frac{2j}{\sigma_1 \dotsm \sigma_{2j+1}} \right)
\leq 0.
\end{align*}

The second assertion is immediate in view of Lemma \ref{commutative diagram} and $(1,2) \not \in f(I)$.
\end{proof}

\begin{lemma}\label{E star is continuous}
The error-sum function $\mathcal{E}^* \colon \Sigma \to \mathbb{R}$ is continuous.
\end{lemma}

\begin{proof}
We showed that the series in \eqref{definition of E star} is uniformly convergent on $\Sigma$. But $\varphi$ is continuous by Lemma \ref{phi is Lipschitz} and $\varphi_n$ is cleary continuous, and so each term in the series of $\mathcal{E}^*$ is continuous. Therefore, $\mathcal{E}^*$ is continuous as a uniformly convergent series of continuous functions.
\end{proof}

The $\lambda$-almost everywhere continuity theorem for $\mathcal{E}(x)$ is now immediate.

\begin{theorem} \label{continuity theorem}
The error-sum function $\mathcal{E} \colon I \to \mathbb{R}$ is continuous on $I \setminus E'$ and so $\mathcal{E}$ is continuous $\lambda$-almost everywhere.
\end{theorem}

\begin{proof}
Let $x \in I \setminus E'$. By Lemma \ref{f is continuous at irrational}, we know that $f \colon I \to \Sigma$ is continuous at $x$. Moreover, $\mathcal{E}^* \colon \Sigma \to \mathbb{R}$ is continuous by Lemma \ref{E star is continuous}. But $\mathcal{E} = \mathcal{E}^* \circ f$ by Lemma \ref{commutative diagram}, and therefore, $\mathcal{E}$ is continuous at $x$.

For the second assertion, it is enough to recall that $E' = \mathbb{Q} \cap (0,1)$ which has zero $\lambda$-measure. Thus $I \setminus E'$ is of full $\lambda$-measure, and the result follows.
\end{proof}

On the other hand, we will show that $\mathcal{E} \colon I \to \mathbb{R}$ fails to be continuous at every point of $E'$  (Theorem \ref{discontinuity theorem}). The following lemma plays a key role in the proof of the theorem.

\begin{lemma} \label{discontinuity lemma}
Let $x \in E'$ and put $\varphi^{-1}( x ) = \{ \sigma, \sigma' \}$, where $\sigma \in \Sigma_{\realizable} \cap \Sigma_n$ and $\sigma' \in (\Sigma \setminus \Sigma_{\realizable}) \cap \Sigma_{n+1}$ for some $n \in \mathbb{N}$. Then
\[
\lim_{\substack{t \to x \\ t \not \in I_{\sigma}}} \mathcal{E}(t) = \mathcal{E}^*(\sigma') = \mathcal{E}(x) + \frac{(-1)^n}{d_1(x) \dotsm d_{n-1}(x) (d_n(x)-1) d_n(x)}.
\]
\end{lemma}

\begin{proof}
By Lemma \ref{commutative diagram}, the continuity of $\mathcal{E}^*$ (Lemma \ref{E star is continuous}), and Lemma \ref{f is discontinuous at rational}, we obtain the first equality as follows:
\[
\lim_{\substack{t \to x \\ t \not \in I_{\sigma}}} \mathcal{E}(t) = \lim_{\substack{t \to x \\ t \not \in I_{\sigma}}} (\mathcal{E}^* \circ f) (t)
= \mathcal{E}^*  \bigg( \lim_{\substack{t \to x \\ t \not \in I_{\sigma}}} f(t) \bigg) = \mathcal{E}^* (\sigma').
\]
Since $x \in E'$, by Proposition \ref{inverse image of phi}(i), $x$ has a finite Pierce expansion of positive length, say $x = [d_1(x), d_2(x), \dotsc, d_n(x)]_P$ for some $n \in \mathbb{N}$. 
Then, since $\varphi (\sigma') = x$, we have
\begin{align*}
\mathcal{E}^* (\sigma')
&= \sum_{k=1}^{n-1} (x - \varphi_k(\sigma')) + (x-\varphi_{n} (\sigma')) + \sum_{k=n+1}^{\infty} (x - \varphi_k(\sigma')). 
\end{align*}
Note that $\sigma' = (d_1(x), \dotsc, d_{n-1}(x), d_n(x)-1, d_n(x)) \in \Sigma_{n+1}$ by Proposition \ref{inverse image of phi}(i). Hence, $s_k(x) = \varphi_k(\sigma')$ for $1 \leq k \leq n-1$ and $x = \varphi_{k+1}(\sigma') = s_k (x)$ for all $k \geq n$. In particular, we have
\[
x- \varphi_n (\sigma') = \varphi_{n+1} (\sigma') - \varphi_n (\sigma') = \frac{(-1)^{n}}{d_1(x) \dotsm d_{n-1}(x) (d_n(x)-1) d_n(x)}.
\]
Thus
\begin{align*}
\mathcal{E}^* (\sigma')
&= \sum_{k=1}^{n-1} (x - s_k(x)) + (x-\varphi_{n} (\sigma')) \\
&= \mathcal{E}(x) + \frac{(-1)^n}{d_1(x) \dots d_{n-1}(x) (d_{n}(x)-1) d_{n}(x)}.
\end{align*}
\end{proof}

Now we are ready to prove that $\mathcal{E}$ is discontinuous at every point of the dense subset $E' \subseteq I$.

\begin{theorem} \label{discontinuity theorem}
Let $x \in E'$ and put $\varphi^{-1}( x ) = \{ \sigma, \sigma' \}$, where $\sigma \in \Sigma_{\realizable} \cap \Sigma_n$ and $\sigma' \in (\Sigma \setminus \Sigma_{\realizable}) \cap \Sigma_{n+1}$ for some $n \in \mathbb{N}$. Write $x = [d_1(x), d_2(x), \dotsc, d_n(x)]_P$. Then the following hold.
\begin{enumerate}[label=(\roman*)]
\item
If $n$ is odd, then $\mathcal{E}$ is left-continuous but has a right jump discontinuity at $x$; more precisely,
\begin{align} \label{right jump formula}
\lim_{t \to x^-} \mathcal{E}(t) &= \mathcal{E}^*(\sigma) = \mathcal{E}(x), \nonumber \\
\lim_{t \to x^+} \mathcal{E}(t) &= \mathcal{E}^*(\sigma') = \mathcal{E}(x) - \frac{1}{d_1(x) \dots d_{n-1}(x) (d_{n}(x)-1) d_{n}(x)}.
\end{align}
\item
If $n$ is even, then $\mathcal{E}$ is right-continuous but has a left jump discontinuity at $x$; more precisely,
\begin{align} \label{left jump formula}
\lim_{t \to x^+} \mathcal{E}(t) &= \mathcal{E}^*(\sigma) = \mathcal{E}(x), \nonumber \\
\lim_{t \to x^-} \mathcal{E}(t) &= \mathcal{E}^*(\sigma') = \mathcal{E}(x) + \frac{1}{d_1(x) \dots d_{n-1}(x) (d_{n}(x)-1) d_{n}(x)}.
\end{align}
\end{enumerate}
\end{theorem}

\begin{proof}
By Proposition \ref{inverse image of phi}(i), we have
\begin{align*}
\sigma &= (d_1(x), \dotsc, d_{n-1}(x), d_n(x)) \in \Sigma_n \cap \Sigma_{\realizable}, \\
\sigma' &= (d_1(x), \dotsc, d_{n-1}(x), d_n(x)-1, d_n(x)) \in \Sigma_{n+1} \cap (\Sigma \setminus \Sigma_{\realizable}).
\end{align*}
Then $\varphi (\sigma) = \varphi (\sigma') = x$ and $\sigma = f(x)$, but $\sigma' \not \in f(I)$.

(i) 
Assume $n$ is odd. Since $I_\sigma = (\varphi (\widehat{\sigma}), \varphi (\sigma)]$ by \eqref{fundamental interval 1} and $x = \varphi (\sigma)$, we have that $t \to x^+$ if and only if $t \to x$ with $t \not \in I_\sigma$. For the right-hand limit, apply Lemma \ref{discontinuity lemma} to obtain \eqref{right jump formula}. For the left-hand limit, note that $t \to x^-$ if and only if $t \to x$ with $t \in I_\sigma \setminus \{ x \}$. Then, by Lemma \ref{commutative diagram}, the continuity of $\mathcal{E}^*$ (Lemma \ref{E star is continuous}), and Lemma \ref{f is discontinuous at rational}, we deduce that
\[
\lim_{t \to x^-} \mathcal{E}(t) = \lim_{t \to x^-} (\mathcal{E}^* \circ f)(t) = \mathcal{E}^* \bigg( \lim_{\substack{t \to x \\ t \in I_\sigma \setminus \{ x \}}} f(t) \bigg) = \mathcal{E}^* (\sigma).
\]
But $\mathcal{E}(x) = (\mathcal{E}^* \circ f)(x) = \mathcal{E}^*(\sigma)$ by Lemma \ref{commutative diagram} and therefore, we conclude that $\mathcal{E}$ is left-continuous at $x$.

(ii)
The proof is similar to that of part (i), so we omit the details.
\end{proof}

Note that for every point $x \in E'$, we have that $\lim_{t \to x^-} \mathcal{E}(t)$ is strictly greater than $\lim_{t \to x^+} \mathcal{E}(t)$, regardless of left or right discontinuity.

The following lemma provides us with the maximum and minimum of $\mathcal{E}^* \colon \Sigma \to \mathbb{R}$ on each cylinder set $\Upsilon_\sigma$. Recall that given $\sigma \coloneqq (\sigma_1, \dotsc, \sigma_{n-1}, \sigma_n) \in \Sigma_n$ for some $n \in \mathbb{N}$, the sequence $\widehat{\sigma}$ is defined as $(\sigma_1, \dotsc, \sigma_{n-1}, \sigma_n+1) \in \Sigma_n$.

\begin{lemma} \label{sup inf in Upsilon sigma}
Let $n \in \mathbb{N}$ and $\sigma \in \Sigma_n$. Then the following hold.
\begin{enumerate}[label=(\roman*)]
\item
If $n$ is odd, we have
\[
\max_{\tau \in \Upsilon_\sigma} \mathcal{E}^*(\tau) = \mathcal{E}^*(\sigma)
\quad \text{and} \quad
\min_{\tau \in \Upsilon_\sigma} \mathcal{E}^*(\tau) = \mathcal{E}^*(\sigma) - n \cdot \lambda (I_\sigma).
\]

\item
If $n$ is even, we have
\[
\max_{\tau \in \Upsilon_\sigma} \mathcal{E}^*(\tau) = \mathcal{E}^*(\sigma) + n \cdot \lambda (I_\sigma)
\quad \text{and} \quad
\min_{\tau \in \Upsilon_\sigma} \mathcal{E}^*(\tau) = \mathcal{E}^*(\sigma).
\]
\end{enumerate}
\end{lemma}

\begin{proof}
Put $\sigma \coloneqq (\sigma_k)_{k \in \mathbb{N}} \in \Sigma_n$. Let $\widehat{\sigma}' \coloneqq (\widehat{\sigma}'_k)_{k \in \mathbb{N}}$, where $\widehat{\sigma}'_{n+1} = \sigma_{n}+1$ and $\widehat{\sigma}'_k = \sigma_k$ for all $k \in \mathbb{N} \setminus \{ n+1 \}$, i.e.,
\[
\widehat{\sigma}' = (\sigma_1, \dotsc, \sigma_{n-1}, \sigma_n, \sigma_n+1, \infty, \infty, \dotsc) \in \Upsilon_\sigma.
\] 
Then $\widehat{\sigma} \in \Sigma_{\realizable}$ and $\widehat{\sigma}' \in \Sigma \setminus \Sigma_{\realizable}$ with $\varphi (\widehat{\sigma}) = \varphi (\widehat{\sigma}')$.

(i)
Assume $n$ is odd. For any $\tau \coloneqq (\tau_k)_{k \in \mathbb{N}} \in \Upsilon_\sigma$, we have $\tau^{(n)} = \sigma$ by definition of the cylinder set, so by using \eqref{E star formula} and Lemma \ref{sequence lemma}, we obtain
\begin{align*}
\mathcal{E}^*(\tau)
&= \sum_{k=1}^{n-1} \frac{(-1)^k k}{\sigma_1 \dotsm \sigma_{k+1}} + \sum_{k=n}^{\infty} \frac{(-1)^k k}{\sigma_1 \dotsm \sigma_{n} \tau_{n+1} \dotsm \tau_{k+1}} \\
&= \mathcal{E}^*(\sigma) - \frac{1}{\sigma_1 \sigma_2 \dotsm \sigma_{n}} \sum_{j=1}^\infty \left( \frac{n+(2j-2)}{\tau_{n+1} \dotsm \tau_{n+(2j-1)}} - \frac{n+(2j-1)}{\tau_{n+1} \dotsm \tau_{n+2j}} \right) 
\leq \mathcal{E}^*(\sigma).
\end{align*}
This shows that $\mathcal{E}^*(\tau) \leq \mathcal{E}^*(\sigma)$ for any $\tau \in \Upsilon_\sigma$, and that $\mathcal{E}^*(\tau)$ attains the maximum when $\tau = \sigma \in \Upsilon_\sigma$.

Again by \eqref{E star formula} and Lemma \ref{sequence lemma}, for any $\tau \coloneqq (\tau_k)_{k \in \mathbb{N}} \in \Upsilon_\sigma$, we have
\begin{align*}
\mathcal{E}^*(\tau)
&= \mathcal{E}^*(\sigma) - \frac{n}{\sigma_1 \sigma_2 \dotsm \sigma_{n} \tau_{n+1}} \\
&\qquad + \frac{1}{\sigma_1 \sigma_2 \dotsm \sigma_{n}} \sum_{j=1}^\infty \left( \frac{n+(2j-1)}{\tau_{n+1} \dotsm \tau_{n+2j}} - \frac{n+2j}{\tau_{n+1} \dotsm \tau_{n+(2j+1)}} \right)  \\
&\geq \mathcal{E}^*(\sigma) - \frac{n}{\sigma_1 \sigma_2 \dotsm \sigma_{n} \tau_{n+1}}
\geq \mathcal{E}^*(\sigma) - \frac{n}{\sigma_1 \sigma_2 \dotsm \sigma_{n} (\sigma_n+1)} 
= \mathcal{E}^*(\sigma) - n \cdot \lambda (I_\sigma),
\end{align*}
where we used $\tau_{n+1} > \sigma_n$ for the second inequality and \eqref{length of I sigma} for the last equality. Notice that the equalities hold if and only if $\tau_{n+1} = \sigma_{n}+1$ and $\tau_{k} = \infty$ for all $k \geq n+2$, i.e., if and only if $\tau = \widehat{\sigma}' \in \Upsilon_\sigma$. Therefore, $\mathcal{E}^*(\tau) \geq \mathcal{E}^*(\sigma) - n \cdot \lambda (I_\sigma)$ for any $\tau \in \Upsilon_\sigma$, and the minimum is attained when $\tau = \widehat{\sigma}' \in \Upsilon_\sigma$.

(ii)
The proof is similar to that of part (i), so we omit the details. 
\end{proof}

Using the preceding lemma, we can describe the supremum and the infimum of $\mathcal{E} \colon I \to \mathbb{R}$ on each fundamental interval $I_\sigma$. We show that approaching the left endpoint from the right yields the infimum, while approaching the right endpoint from the left yields the supremum. (See Proposition \ref{I sigma} for the left and right endpoints of the fundamental intervals.)

\begin{lemma} \label{sup inf in fundamental interval}
Let $n \in \mathbb{N}$ and $\sigma \in \Sigma_n$. Then the following hold.
\begin{enumerate}[label=(\roman*)]
\item
If $n$ is odd, we have
\[
\sup_{t \in I_\sigma} \mathcal{E}(t) = \lim_{t \to (\varphi(\sigma))^-} \mathcal{E}(t)
\quad \text{and} \quad
\inf_{t \in I_\sigma} \mathcal{E}(t) = \lim_{t \to (\varphi(\widehat{\sigma}))^+} \mathcal{E}(t).
\]

\item
If $n$ is even, we have
\[
\sup_{t \in I_\sigma} \mathcal{E}(t) = \lim_{t \to (\varphi(\widehat{\sigma}))^-} \mathcal{E}(t)
\quad \text{and} \quad
\inf_{t \in I_\sigma} \mathcal{E}(t) = \lim_{t \to (\varphi(\sigma))^+} \mathcal{E}(t).
\]
\end{enumerate}
\end{lemma}

\begin{proof}
(i)
Assume $n$ is odd. Put $\sigma \coloneqq (\sigma_k)_{k \in \mathbb{N}} \in \Sigma_n$. Then $\widehat{\sigma} = (\sigma_1, \dotsc, \sigma_{n-1}, \sigma_n+1) \in \Sigma_n$, and by Proposition \ref{inverse image of phi}(i), we have $\varphi (\widehat{\sigma}) = \varphi (\widehat{\sigma}')$, where $\widehat{\sigma}' \coloneqq (\sigma_1, \dotsc, \sigma_{n-1}, \sigma_n, \sigma_n+1) \in \Sigma_{n+1}$ with $\widehat{\sigma} \in \Sigma_{\realizable}$ and $\widehat{\sigma}' \in \Sigma \setminus \Sigma_{\realizable}$.

By using $\mathcal{E} = \mathcal{E}^* \circ f$ (Lemma \ref{commutative diagram}), $\overline{f(I_\sigma)} = \Upsilon_\sigma$ (Lemma \ref{f image of fundamental interval is dense}), and the continuity of $\mathcal{E}^*$ (Lemma \ref{E star is continuous}), we find that
\begin{align}
\sup_{t \in I_\sigma} \mathcal{E}(t) = \sup_{f(t) \in f(I_\sigma)} \mathcal{E}^*(f(t)) = \sup_{\tau \in \Upsilon_\sigma} \mathcal{E}^*(\tau) &= \max_{\tau \in \Upsilon_\sigma} \mathcal{E}^*(\tau) = \mathcal{E}^*(\sigma), \label{sup in I sigma}  \\
\inf_{t \in I_\sigma} \mathcal{E}(t) = \inf_{f(t) \in f(I_\sigma)} \mathcal{E}^*(f(t)) = \inf_{\tau \in \Upsilon_\sigma} \mathcal{E}^*(\tau) &= \min_{\tau \in \Upsilon_\sigma} \mathcal{E}^*(\tau) = \mathcal{E}^*(\widehat{\sigma}') \label{inf in I sigma},
\end{align}
where the last two equalities for both \eqref{sup in I sigma} and \eqref{inf in I sigma} follow from Lemma \ref{sup inf in Upsilon sigma} and its proof. 

For the supremum, notice that, by Proposition \ref{I sigma}, $t \to (\varphi (\sigma))^-$ if and only if $t \to \varphi (\sigma)$ with $t \in I_\sigma \setminus \{ \varphi (\sigma) \}$. Then by Lemma \ref{commutative diagram}, the continuity of $\mathcal{E}^*$ (Lemma \ref{E star is continuous}), and Lemma \ref{f is discontinuous at rational}, we obtain
\[
\lim_{t \to (\varphi (\sigma))^-} \mathcal{E}(t) = \lim_{\substack{t \to \varphi (\sigma) \\ t \in I_{\sigma} \setminus \{ \varphi (\sigma) \}}} \mathcal{E}(t) = \mathcal{E}^* \bigg( \lim_{\substack{t \to \varphi (\sigma) \\ t \in I_{\sigma} \setminus \{ \varphi (\sigma) \}}} f(t) \bigg)  =  \mathcal{E}^*(\sigma).
\]
Combining this with \eqref{sup in I sigma} gives the result.

For the infimum, notice that, by Proposition \ref{I sigma}, $t \to (\varphi (\widehat{\sigma}))^+$ if and only if $t \to \varphi (\widehat{\sigma})$ with $t \not \in I_{\widehat{\sigma}}$. Hence Lemma \ref{discontinuity lemma} tells us that
\[
\lim_{t \to (\varphi(\widehat{\sigma}))^+} \mathcal{E}(t) = \lim_{\substack{t \to \varphi(\widehat{\sigma}) \\ t \not \in I_{\widehat{\sigma}}}} \mathcal{E}(t) = \mathcal{E}^* (\widehat{\sigma}').
\]
Combining this with \eqref{inf in I sigma} gives the result.

(ii)
The proof is similar to that of part (i), so we omit the details. 
\end{proof}

The following lemma is an analogue of \cite[Lemma~2.6]{SW07}, where the error-sum function of L{\"u}roth series is discussed. This lemma will serve as the key ingredient in finding a suitable covering for the graph of $\mathcal{E}(x)$ in Section \ref{Section 4}.

\begin{lemma} \label{fourth lemma}
Let $n \in \mathbb{N}$ and $\sigma \in \Sigma_n$. Then
\[
\sup_{t,u \in I_\sigma} |\mathcal{E}(t) - \mathcal{E}(u)| = n \cdot \lambda (I_\sigma).
\]
\end{lemma}

\begin{proof}
By Lemmas \ref{commutative diagram}, \ref{f image of fundamental interval is dense}, and \ref{E star is continuous}, we have, as in \eqref{sup in I sigma} and \eqref{inf in I sigma},
\[
\sup_{t,u \in I_\sigma} |\mathcal{E}(t) - \mathcal{E}(u)|
= \sup_{t \in I_\sigma} \mathcal{E}(t) - \inf_{t \in I_\sigma} \mathcal{E}(t) 
= \max_{\tau \in \Upsilon_\sigma} \mathcal{E}^*(\tau) - \min_{\tau \in \Upsilon_\sigma} \mathcal{E}^*(\tau).
\]
Notice that the last term equals $n \cdot \lambda (I_\sigma)$ by Lemma \ref{sup inf in Upsilon sigma}. This proves the lemma.
\end{proof}

One might be tempted to say that $\mathcal{E} \colon I \to \mathbb{R}$ is fairly regular in the sense of $\lambda$-almost everywhere continuity (Theorem \ref{continuity theorem}). However, the following theorem tells us that $\mathcal{E}$ is not well-behaved in the bounded variation sense.

\begin{theorem} \label{bounded variation theorem}
The error-sum function $\mathcal{E} \colon I \to \mathbb{R}$ is not of bounded variation.
\end{theorem}

\begin{proof}
Let $V_I (\mathcal{E})$ denote the total variation of $\mathcal{E}$ on $I$. Let $n \in \mathbb{N}$. We consider the collection $\mathcal{I} \coloneqq \{ I_\sigma : \sigma \in \Sigma_n \}$, i.e., the collection of all fundamental intervals of order $n$. Note that $\sum_{\sigma \in \Sigma_n} \lambda (I_\sigma) = 1$. Then, by Lemma \ref{fourth lemma}, we have
\[
n = n \sum_{\sigma \in \Sigma_n} \lambda (I_\sigma) = \sum_{\sigma \in \Sigma_n} \sup_{t, u \in I_\sigma} |\mathcal{E}(t) - \mathcal{E}(u)| \leq V_I(\mathcal{E}),
\]
where the inequality follows from the fact that the $I_\sigma \in \mathcal{I}$ are mutually disjoint intervals. Since $n \in \mathbb{N}$ is arbitrary, it follows that $V_I(\mathcal{E})$ is not finite. This completes the proof.
\end{proof}

We prove that $\mathcal{E} \colon I \to \mathbb{R}$ enjoys an intermediate value property in some sense, which is an analogue of \cite[Theorem~4.3]{RP00}. A similar result can also be found in \cite[Theorem~2.5]{SMZ06}. In fact, every result aforementioned is a consequence of the following theorem.

\begin{theorem} \label{IVT theorem}
Suppose that $g \colon J \to \mathbb{R}$ is a function on an interval $J \subseteq \mathbb{R}$ satisfying the following conditions.
\begin{enumerate}[label=(\roman*)]
\item
There exists a subset $D$ of the interior of $J$ such that $g$ is continuous on $J \setminus D$.
\item
For any $x \in D$, $g$ is either left-continuous or right-continuous at $x$ with $\lim_{t \to x^-} g(t) > \lim_{t \to x^+} g(t)$.
\end{enumerate}
Let $a,b \in J$ with $a<b$. If $g(a) < y < g(b)$, then there exists an $x \in (a,b) \setminus D$ such that $g(x) = y$.
\end{theorem}

\begin{proof}
Consider the set
\[
S \coloneqq \{ t \in [a,b] : g (t) < y \},
\]
and let $t_0 \coloneqq \sup S$. Since $g(a) < y$ by assumption, we have $a \in S$, and hence $S$ is non-empty. So $t_0 \neq - \infty$ and $t_0 \in [a,b]$. Our aim is to show that $t_0$ is a desired root, that is, $g(t_0) = y$ and $t_0 \in (a,b) \setminus D$.

We claim that $t_0 > a$. We consider three cases depending on the continuity at $a$.

{\sc Case I}.
Assume $a \in J \setminus D$, so that $g$ is continuous at $a$ by condition (i). Then, since $g(a) < y$, there is an $\eta_1 \in (0, b-a)$ such that $g(t) < y$ for all $t \in (a-\eta_1, a+\eta_1) \cap J$. So $t_0 \geq a+\eta_1$ and hence $t_0 > a$.

{\sc Case II}.
Assume that $a \in D$ and $g$ is left-continuous at $a$. Then $\lim_{t \to a^+} g(t) < \lim_{t \to a^-} g(t) = g(a) < y$ by condition (ii) and assumption. 
By definition of the right-hand limit, there exists an $\eta_2 \in (0, b-a)$ such that $g (t) < y$ for all $t \in (a, a+ \eta_2)$. So $t_0 \geq a + \eta_2$ and hence $t_0 > a$.

{\sc Case III}.
Assume that $a \in D$ and $g$ is right-continuous at $a$. Then $\lim_{t \to a^+} g(t) = g(a) < y$ by assumption. 
By definition of the right-hand limit, there exists an $\eta_3 \in (0, b-a)$ such that $g (t) < y$ for all $t \in (a, a + \eta_3)$. So $t_0 \geq a + \eta_3$ and hence $t_0 > a$.

By a similar argument, which we omit here, we can show that $t_0 < b$.

We have shown above that $t_0 \in (a,b)$. It remains to prove that $g (t_0) = y$ with $t_0 \not \in D$. We show first that $t_0 \not \in D$. Suppose $t_0 \in D$ to argue by contradiction. Since $t_0 = \sup S$ we can find a sequence $( a_n )_{n \in \mathbb{N}}$ in $S$ such that $a_n \leq t_0$ for each $n \in \mathbb{N}$ and $a_n \to t_0$ as $n \to \infty$. (We can choose $a_n \in S$ such that $t_0 - 1/n < a_n \leq t_0$ for each $n \in \mathbb{N}$.) Similarly, we can find a sequence $( b_n )_{n \in \mathbb{N}}$ in $[a,b] \setminus S$ such that $b_n \geq t_0$ for each $n \in \mathbb{N}$ and $b_n \to t_0$ as $n \to \infty$. Then, by our choice of two sequences, $g(a_n) < y$ and $g(b_n) \geq y$ for all $n \in \mathbb{N}$. Now note that since $t_0 \in D$, $g$ is either left-continuous or right-continuous at $t_0$ by condition (ii). If $g$ is left-continuous at $t_0$, then by condition (ii), we have
\[
y \geq \lim_{n \to \infty} g(a_n) = g (t_0) > \lim_{n \to \infty} g(b_n) \geq y,
\]
which is a contradiction. If $g$ is right-continuous at $t_0$, then by condition (ii), we have
\[
y \geq \lim_{n \to \infty} g (a_n) > g (t_0) = \lim_{n \to \infty} g (b_n) \geq y,
\]
which is a contradiction. This proves that $t_0 \not \in D$, as desired. 

Since $t_0 \in J \setminus D$, we know from condition (i) that $g$ is continuous at $t_0$. Hence, $g(t_0) = y$ by definitions of $S$ and $t_0$. For, if not, say $g(t_0) < y$, we can find a $\delta \in (0, \min \{ t_0-a, b-t_0 \})$ such that $g(t)<y$ on the interval $(t_0-\delta, t_0+\delta)$, which contradicts $t_0 = \sup S$. Similarly, $g(t_0) > y$ gives a contradiction. This completes the proof that $t_0 \in (a,b) \setminus D$ is a root of $g(x) = y$ we were seeking.
\end{proof}

\begin{remark}
In Theorem \ref{IVT theorem}, the assumption $g(a) < y < g(b)$ for $a<b$ is stricter than that of the standard Intermediate Value Theorem in $\mathbb{R}$. This additional assumption is necessary because at every discontinuity, $g$ has a sudden drop therein. To be precise, for every $x \in D$, we have $\lim_{t \to x^-}g(t) > \lim_{t \to x^+} g(t)$ by condition (ii). For the same phenomenon for $\mathcal{E} \colon I \to \mathbb{R}$, see Theorem \ref{discontinuity theorem} and equations \eqref{right jump formula} and \eqref{left jump formula} therein.
\end{remark}

\begin{corollary} \label{IVT}
Let $a,b \in I$ with $a<b$. If $\mathcal{E}(a) < y < \mathcal{E}(b)$, then there exists an $x \in (a,b) \setminus E$ such that $\mathcal{E}(x) = y$.
\end{corollary}

\begin{proof}
By Theorems \ref{continuity theorem} and \ref{discontinuity theorem}, $\mathcal{E}$ satisfies the two conditions of Theorem \ref{IVT theorem} with $J \coloneqq I$ and $D \coloneqq E'$. Since $(a,b) \setminus E = (a,b) \setminus E'$, the result follows from Theorem \ref{IVT theorem}.
\end{proof}

Using Theorem \ref{IVT theorem}, we can prove the intermediate value property of $P \colon \mathbb{R} \to \mathbb{R}$, the error-sum function of the regular continued fraction expansion, defined as in Section \ref{Section 1}. Compare the following corollary with \cite[Theorem~4.3]{RP00}, where the authors considered $P|_I$, the restriction of $P$ to $I$.

\begin{corollary}
Let $a, b \in \mathbb{R}$ with $a<b$. If $P (a) < y < P (b)$, then there exists an $x \in (a,b) \setminus \mathbb{Q}$ such that $P (x) = y$.
\end{corollary}

\begin{proof}
Let $x$ be rational. Then the regular continued fraction expansion of $x$ is of finite length, say $x = [a_0(x); a_1(x), \dotsc, a_n(x)]$ for some $n \in \mathbb{N}_0$. By \cite[Lemma~1.1]{RP00} and \cite[Theorem~2.3]{RP00}, the following hold:
\begin{enumerate}[label=(\roman*)]
\item
If $n$ is odd, then $P$ is left-continuous but has a right jump discontinuity at $x$ with the right-hand limit $\lim_{t \to x^+} P(t) = P(x) - 1/q_n(x)$.
\item
If $n$ is even, then $P$ is right-continuous but has a left jump discontinuity at $x$ with the left-hand limit $\lim_{t \to x^-} P(t) = P(x) + 1/q_n(x)$.
\end{enumerate}
Since $q_n(x) >0$ by definition (see \cite{RP00}, p. 274), it follows that $\lim_{t \to x^-} P(x) > \lim_{t \to x^+} P(x)$ for every $x \in \mathbb{Q}$.
Moreover, by Theorem 2.3 of \cite{RP00}, $P$ is continuous at every irrational point. Therefore, by taking $J \coloneqq \mathbb{R}$ and $D \coloneqq \mathbb{Q}$ in Theorem \ref{IVT theorem}, the result follows.
\end{proof}

We showed that $\mathcal{E}$ is bounded on $I$ (Theorem \ref{boundedness theorem}) and it is continuous $\lambda$-almost everywhere (Theorem \ref{continuity theorem}). Hence, $\mathcal{E}$ is Riemann integrable on $I$. Before calculating the integral, we first find a useful formula for $\mathcal{E}$.

\begin{lemma} \label{E formula lemma}
For every $x \in I$ and for each $n \in \mathbb{N}$, we have
\begin{align} \label{simplified formula}
\mathcal{E}(x) = \sum_{k=1}^n ( x - s_k(x) ) + \frac{(-1)^n}{d_1(x) \dotsm d_n(x)} \mathcal{E} (T^n x).
\end{align}
\end{lemma}

\begin{proof}
Let $x \in I$ and $n \in \mathbb{N}$. From the definition of digits, we have $d_{n+j}(x) = d_j (T^nx)$ for any $j \in \mathbb{N}$. Then by making use of \eqref{simplified series}, we obtain
\begin{align*}
\mathcal{E}(x) - \sum_{k=1}^n ( x - s_k(x) )
&= \sum_{k=n+1}^\infty ( x - s_k(x) ) \\
&= \sum_{k=n+1}^\infty \frac{(-1)^k T^kx}{d_1(x) \dotsm d_n(x) d_{n+1}(x) \dotsm d_k(x)} \\
&= \frac{(-1)^n}{d_1(x) \dotsm d_n(x)} \sum_{j=1}^\infty \frac{(-1)^jT^{n+j} x}{d_{n+1}(x) \dotsm d_{n+j}(x)} \\
&= \frac{(-1)^n}{d_1(x) \dotsm d_n(x)} \sum_{j=1}^\infty \frac{(-1)^j T^j (T^nx)}{d_1(T^nx) \dotsm d_j(T^nx)} \\
&= \frac{(-1)^n}{d_1(x) \dotsm d_n(x)} \sum_{j=1}^\infty ( T^nx - s_j (T^nx) ) \\
&= \frac{(-1)^n}{d_1(x) \dotsm d_n(x)} \mathcal{E} (T^n x),
\end{align*}

as desired.
\end{proof}

\begin{theorem}
We have
\[
\int_0^1 \mathcal{E}(x) \, dx = - \frac{1}{8}.
\]
\end{theorem}

\begin{proof}
Note that on the interval $( 1/2, 1 ]$, we have $d_1(x) = 1$ and $s_1(x) = 1$, and hence $\mathcal{E}(x) = x-1 - \mathcal{E}(Tx)$ by \eqref{simplified formula}. By letting $Tx = u = 1-x$ so that $du = - dx$ on the interval $( 1/2, 1 ]$, we obtain
\begin{align*}
\int_0^1 \mathcal{E}(x) \, dx
&= \int_0^{1/2} \mathcal{E}(x) \, dx + \int_{1/2}^1 ( x-1 - \mathcal{E}(Tx) ) \, dx \\
&= \int_0^{1/2} \mathcal{E}(x) \, dx + \int_{1/2}^1 (x-1) \, dx - \int_0^{1/2} \mathcal{E}(u) \, du 
= - \frac{1}{8}.
\end{align*}
\end{proof}

Before we move on to the next section, we prove one lemma which will be used in Section \ref{Section 4}.

\begin{lemma} \label{epsilon cover}
Let $\sigma \coloneqq (\sigma_k)_{k \in \mathbb{N}} \in \Sigma$. For any $n \in \mathbb{N}$, we have
\[
\big| \varphi (\sigma) - \varphi (\sigma^{(n)}) \big| \leq \frac{1}{\sigma_1 \dotsm \sigma_n \sigma_{n+1}}
\quad
\text{and}
\quad
\big| \mathcal{E}^* (\sigma) - \mathcal{E}^* (\sigma^{(n)}) \big| \leq \frac{n}{\sigma_1 \dotsm \sigma_n \sigma_{n+1}}. 
\]
\end{lemma}

\begin{proof}
Let $n \in \mathbb{N}$. The first inequality is immediate from the definitions of $\varphi$ and $\sigma^{(n)}$. Indeed, since $\sigma$ and $\sigma^{(n)}$ share the initial block of length $n$, by \eqref{phi - phi n}, we have
\begin{align*}
\big| \varphi (\sigma) - \varphi (\sigma^{(n)}) \big| 
&= \big| \varphi (\sigma) - \varphi_n (\sigma) \big|
\leq \frac{1}{\sigma_1 \dotsm \sigma_n \sigma_{n+1}}.
\end{align*}

The second inequality follows from Lemma \ref{E star formula lemma}. For $\mathcal{E}^*(\sigma^{(n)})$, we just need to take $\sigma_k = \infty$ for all $k \geq n+1$ in the formula \eqref{E star formula} to obtain
\begin{align*}
\mathcal{E}^* (\sigma^{(n)}) = \sum_{k=1}^{n-1} \frac{(-1)^k k}{\sigma_1 \dotsm \sigma_{k+1}}.
\end{align*}
Thus, by Lemma \ref{sequence lemma}, we find that
\begin{align*}
\big| \mathcal{E}^* (\sigma) - \mathcal{E}^* (\sigma^{(n)}) \big|
&= \left| \sum_{k=n}^\infty \frac{(-1)^k k}{\sigma_1 \dotsm \sigma_{k+1}} \right| \\
&= \left| \frac{n}{\sigma_1 \dotsm \sigma_{n+1}} - \sum_{j=1}^\infty \left( \frac{n+(2j-1)}{\sigma_1 \dotsm \sigma_{n+2j}} - \frac{n+2j}{\sigma_1 \dotsm \sigma_{n+(2j+1)}} \right) \right|
\leq \frac{n}{\sigma_1 \dotsm \sigma_{n+1}}.
\end{align*}
\end{proof}

\section{The dimension of the graph of $\mathcal{E}(x)$} \label{Section 4}

In this section, we determine three widely used and well-known dimensions, namely the Hausdorff dimension, the box-counting dimension, and the covering dimension, of the graph of the error-sum function $\mathcal{E} \colon I \to \mathbb{R}$. In fact, although $\mathcal{E}$ is discontinuous on a dense subset of $I$ (Theorem \ref{discontinuity theorem}) and is not of bounded variation (Theorem \ref{bounded variation theorem}), it is not sufficiently irregular to have a graph of any dimension strictly greater than one. Nevertheless, we show that the Hausdorff dimension of the graph is strictly greater than its covering dimension. This will lead to the conclusion that the graph is indeed a fractal according to Mandelbrot's definition in his prominent book \cite{Man82}, where he coined the term {\em fractal} in a Euclidean space and defined it as a set whose covering dimension is strictly less than its Hausdorff dimension.

Throughout this section, for a subset $F$ of $\mathbb{R}$ or of $\mathbb{R}^2$, we denote by $\mathcal{H}^s (F)$ the $s$-dimensional Hausdorff measure of $F$ and by $\hdim F$ the Hausdorff dimension of $F$. In addition, we denote by $\lbdim F$ and $\ubdim F$ the lower and upper box-counting dimension of $F$, respectively. If $\lbdim F = \ubdim F$, we call this common value the box-counting dimension of $F$ and denote the value by $\bdim F$. Lastly, the covering dimension of $F$ is denoted by $\covdim F$.

We refer the reader to \cite[Chapters~2--4]{Fal14} for details on the Hausdorff measure, the Hausdorff dimension, and the box-counting dimension, and \cite[Chapters~1--2]{Coo15} for the covering dimension which is called the topological dimension in the book.

\subsection{The Hausdorff dimension of the graph of $\mathcal{E}(x)$}

Define $G \colon I \to I \times \mathbb{R}$ by $G(x) \coloneqq (x, \mathcal{E}(x))$ for $x \in I$. Then $G(I) = \{ ( x, \mathcal{E}(x) ) : x \in I \}$ is the graph of $\mathcal{E}$.

It should be mentioned that the proof idea of the following theorem is borrowed from earlier studies, e.g., \cite{CWY14, RP00, SMZ06, SW07}.

\begin{theorem} \label{Hausdorff dimension of graph}
The graph of the error-sum function $\mathcal{E} \colon I \to \mathbb{R}$ has the Hausdorff dimension one, i.e., $\hdim G(I)=1$.
\end{theorem}

\begin{proof}
For the lower bound, we use the projection map onto the first coordinate $\Proj \colon \mathbb{R}^2 \to \mathbb{R}^2$ given by $\Proj ( (x,y) ) = (x,0)$ for each $(x,y) \in \mathbb{R}^2$. Recall the formula (6.1) in \cite{Fal14} which tells us that for any subset $F$ of $\mathbb{R}^2$ we have $\hdim \Proj (F) \leq \min \{ \hdim F, 1 \}$. It follows that
\[
1 = \hdim I = \hdim ( \Proj (G(I)) ) \leq \hdim G(I).
\]

For the upper bound, we find a suitable covering for $G(I)$. For any $n \in \mathbb{N}$ and $\sigma \in \Sigma_n$, we define a closed interval $J_\sigma \subseteq \mathbb{R}$ by
\[
J_\sigma \coloneqq 
\left[ \inf_{t \in I_\sigma} \mathcal{E}(t), \sup_{t \in I_\sigma} \mathcal{E}(t) \right].
\]
Then $I_\sigma \times \mathcal{E}(I_\sigma) \subseteq I_\sigma \times J_\sigma$. We claim that, for any $n \in \mathbb{N}$, we have $I \setminus E \subseteq \bigcup_{\sigma \in \Sigma_n} I_\sigma$, where $E = I \cap \mathbb{Q}$. Indeed, if $x \in I \setminus E$, then $\tau \coloneqq f(x) \in \Sigma_\infty$ by Proposition \ref{inverse image of phi}(ii). Clearly, $f(x) \in \Upsilon_{\tau^{(n)}}$ with $\tau^{(n)} \in \Sigma_n$. Hence, $x \in f^{-1} (\Upsilon_{\tau^{(n)}}) = I_{\tau^{(n)}} \subseteq \bigcup_{\sigma \in \Sigma_n} I_\sigma$, and this proves the claim. It follows that, for any $n \in \mathbb{N}$, the collection $\mathcal{J} \coloneqq \{ I_\sigma \times J_\sigma : \sigma \in \Sigma_n \}$ is a covering of $F \coloneqq G(I) \setminus G(E)$. Here, we have $\sum_{\sigma \in \Sigma_n} \lambda (I_\sigma) = 1$ from the first coordinate since $\sum_{\sigma \in \Sigma_n} \lambda (I_\sigma) = \lambda \left( \bigcup_{\sigma \in \Sigma_n} I_\sigma \right) \geq \lambda (I \setminus E) = 1$. Due to Lemma \ref{fourth lemma}, we have $\lambda (J_\sigma) = n \cdot \lambda (I_\sigma)$. So we can cover each $I_\sigma \times J_\sigma \in \mathcal{J}$ by a rectangle of base $\lambda (I_\sigma)$ and height $n \cdot \lambda (I_\sigma)$. Note that such rectangles have diameter $\sqrt{n^2+1} \cdot \lambda (I_\sigma)$. Let $\varepsilon >0$ be given. Recall from \eqref{bound for length of I sigma} that $\lambda (I_\sigma) \leq 1/(n+1)!$. Then
\begin{align*}
\mathcal{H}^{1+\varepsilon} ( F )
&\leq \liminf_{n \to \infty} \left[ \sum_{\sigma \in \Sigma_n} ( \sqrt{n^2+1} \cdot \lambda (I_\sigma) )^{1+\varepsilon} \right] \\
&\leq \liminf_{n \to \infty} \left[ (\sqrt{n^2+1})^{1+\varepsilon} \left( \frac{1}{(n+1)!} \right)^{\varepsilon} \sum_{\sigma \in \Sigma_n} \lambda (I_\sigma) \right]
= 0.
\end{align*}
The above calculation shows that $\hdim F \leq 1+\varepsilon$. Since $\varepsilon>0$ was arbitrary, it follows that $\hdim F \leq 1$. Now note that the set $G(E) = \bigcup_{x \in E} \{ G(x) \}$ is countable as a union of singletons over a countable set $E = I \cap \mathbb{Q}$. Therefore, by countable stability of the Hausdorff dimension (the third property in \cite[pp.~48--49]{Fal14}), we deduce that $\hdim G(I) = \sup_{x \in E} \{ \hdim F,  \hdim \{ G(x) \} \} \leq 1$, and this completes the proof.
\end{proof}

\subsection{The box-counting dimension of the graph of $\mathcal{E}(x)$}

We shall establish the following theorem.

\begin{theorem} \label{box dimension of graph}
The graph of the error-sum function $\mathcal{E} \colon I \to \mathbb{R}$ has the box-counting dimension one, i.e., $\bdim G(I)=1$.
\end{theorem}

We define $\Gamma \colon \Sigma \to I \times \mathbb{R}$ by $\Gamma (\sigma) \coloneqq ( \varphi (\sigma), \mathcal{E}^* (\sigma) )$ for $\sigma \in \Sigma$.

\begin{lemma} \label{properties of Gamma}
The following two properties hold for $\Gamma$:
\begin{enumerate}[label=(\roman*)]
\item
$\Gamma \colon \Sigma \to \Gamma (\Sigma)$ is a homeomorphism.
\item
$\Gamma (\Sigma)$ is compact.
\end{enumerate}
\end{lemma}

\begin{proof}
(i)
It is enough to show that $\Gamma$ is a continuous injection, since $\Sigma$ is compact (Lemma \ref{Sigma is closed}) and $\Gamma (\Sigma) \subseteq \mathbb{R}^2$ is Hausdorff. Since $\varphi \colon \Sigma \to I$ and $\mathcal{E}^* \colon \Sigma \to \mathbb{R}$ are continuous by Lemmas \ref{phi is Lipschitz} and \ref{E star is continuous}, respectively, it follows that $\Gamma$ is continuous. 

To prove injectivity, suppose $\Gamma (\sigma) = \Gamma (\tau)$. Then $\varphi (\sigma) = \varphi (\tau)$ from the first coordinate. Assume $\sigma \neq \tau$. Then, by Proposition \ref{inverse image of phi}(i), we have either $\sigma \in \Sigma_{\realizable}$ with $\tau \not \in \Sigma_{\realizable}$ or $\tau \in \Sigma_{\realizable}$ with $\sigma \not \in \Sigma_{\realizable}$. In either case, Theorem \ref{discontinuity theorem} tells us that $\mathcal{E}^*(\sigma) \neq \mathcal{E}^*(\tau)$, which is a contradiction. Thus $\sigma = \tau$.

(ii)
Since $\Sigma$ is compact by Lemma \ref{Sigma is closed}, the result follows from part (i).
\end{proof}

\begin{lemma} \label{Gamma is a closure of G}
We have $\Gamma (\Sigma) = \overline{G(I)}$.
\end{lemma}

\begin{proof}
First note that since $\varphi \circ f = \id_{I}$ and $\mathcal{E}^* \circ f = \mathcal{E}$ (Lemma \ref{commutative diagram}), we have
\[
(\Gamma \circ f) (x) = ( \varphi(f(x)), \mathcal{E}^*(f(x)) ) = (x, \mathcal{E}(x))
\]
for any $x \in I$, and hence $(\Gamma \circ f) (I) = G(I)$. Since $\Sigma$ is compact (Lemma \ref{Sigma is closed}), the continuity of $\Gamma$ (Lemma \ref{properties of Gamma}(i)) tells us that $\Gamma (\overline{f (I)}) = \overline{\Gamma (f (I))}$. Then, by Lemma \ref{closure of f I}, we have
\[
\Gamma (\Sigma) = \Gamma (\overline{f (I)}) = \overline{\Gamma (f (I))} = \overline{G(I)}.
\]
\end{proof}

The following proposition gives us a general relation among $\hdim$, $\lbdim$, and $\ubdim$ for certain subsets of $\mathbb{R}^2$.

\begin{proposition}[{\cite[Proposition~3.4]{Fal14}}] \label{hdim lbdim ubdim}
If $F \subseteq \mathbb{R}^2$ is non-empty and bounded, then 
\[
\hdim F \leq \lbdim F \leq \ubdim F.
\]
\end{proposition}

To prove Theorem \ref{box dimension of graph}, we first find a lower bound for the lower box-counting dimension.

\begin{lemma} \label{lower box dimension}
We have $\lbdim \overline{G(I)} \geq 1$.
\end{lemma}

\begin{proof}
By Proposition \ref{hdim lbdim ubdim}, we have $\lbdim \overline{G(I)} \geq \hdim \overline{G(I)}$. By monotonicity of the Hausdorff dimension and by Theorem \ref{Hausdorff dimension of graph}, we further have $\hdim \overline{G(I)} \geq \hdim G(I) = 1$. Combining the inequalities, the result follows.
\end{proof}

We need the following proposition to find an upper bound for the upper box-counting dimension. The lemma provides an upper bound for the number of finite sequences whose length and the product of all terms are dominated, respectively, by prescribed numbers. The logarithm without base, denoted $\log$, will always mean the natural logarithm.

\begin{proposition}[{\cite[Claim~3]{Luc97}}] \label{Luczak lemma}
Let $p, m \in \mathbb{N}$. Denote by $S(p,m)$ the set of sequences $(\sigma_j)_{j=1}^k \in \mathbb{N}^k$ of finite length $k$ such that $1 \leq k \leq m$ and $\prod_{j=1}^k \sigma_j \leq p$, i.e.,
\[
S(p,m) \coloneqq \left\{ (\sigma_j)_{j=1}^k \in \mathbb{N}^k : 1 \leq k \leq m \text{ and } \prod_{j=1}^k \sigma_j \leq p \right\}.
\]
Then
\[
|S(p,m)| \leq p (2+\log p)^{m-1}.
\]
\end{proposition}

We use Proposition \ref{Luczak lemma} to obtain an upper bound for the number of all strictly increasing finite sequences of fixed length whose product of digits is bounded from above by a given number.

\begin{lemma} \label{refined bound lemma}
Let $p, m \in \mathbb{N}$. Let $I(p,m)$ be defined by
\[
I(p,m) \coloneqq \left\{ (\sigma_j)_{j=1}^m \in \mathbb{N}^m : \sigma_1 < \sigma_2 < \dotsb < \sigma_m \text{ and } \prod_{j=1}^m \sigma_j \leq p \right\},
\]
i.e., the set $I(p,m)$ consists of all finite sequences of positive integers of length $m$ whose terms are strictly increasing and the product of all terms is at most $p$. Then, we have
\[
|I(p,m)| \leq \frac{p (2+\log p)^{m-1}}{m!}.
\]
\end{lemma}

\begin{proof}
Let $S(p,m)$ be as in Proposition \ref{Luczak lemma}. Obviously, $I(p,m) \subseteq S(p,m)$. For any $(\sigma_j)_{j=1}^m \in I(p,m)$, all the terms $\sigma_j$, $1 \leq j \leq m$, are distinct, so that there are $m!$ ways to form a sequence of length $m$ with the same terms. It is clear that all the sequences formed so are members of $S(p,m)$. Thus $m! |I(p,m)| \leq |S(p,m)|$. Therefore, the desired upper bound for $|I(p,m)|$ follows from Proposition \ref{Luczak lemma}.
\end{proof}

The following inequalities are well-known lower and upper bounds for the factorial function. These bounds are rougher than the famous Stirling's formula, but the proof is elementary and they are satisfactory enough in our argument.

\begin{proposition}[{\cite[Lemma~10.1]{KL16}}] \label{factorial bound}
For every $n \in \mathbb{N}$, we have
\begin{align*}
\frac{n^n}{e^{n-1}} \leq n! \leq \frac{n^{n+1}}{e^{n-1}}. 
\end{align*}
\end{proposition}

\begin{proof}
The core idea of the proof is the fact that the map $x \mapsto \log x$ is increasing on $(0, \infty)$. We refer the interested readers to \cite{KL16} in which the detailed proof is given.
\end{proof}

Now we are in a position to establish the upper bound by considering a suitable covering of $\Gamma (\Sigma)$.

\begin{lemma} \label{upper box dimension}
We have $\ubdim \overline{G(I)} \leq 1$.
\end{lemma}

\begin{proof}
Let $\varepsilon \coloneqq 2 e^{-M}$ with $M > 0$ large enough. Take $n = n(M) \in \mathbb{N}$ such that $(n-1)! \leq e^M \leq n!$. Clearly, $n \to \infty$ as $M \to \infty$ and vice versa. Then for any $(\sigma_k)_{k \in \mathbb{N}} \in \Sigma$, by \eqref{sigma n bound}, we have
\begin{align} \label{bound for product of n+1 digits}
\sigma_1 \sigma_2 \dotsm \sigma_n \sigma_{n+1} \geq (n+1)! \geq (n+1)  e^M.
\end{align}
We obtain lower and upper bounds for $M$ by means of Proposition \ref{factorial bound}:
\begin{align} \label{bounds for M}
(n-1) \log (n-1) - (n-2) \leq M \leq (n+1) \log n - (n-1). 
\end{align}
Since $(n-1)!/e^n \to \infty$ as $n \to \infty$ but $(n-1)!/e^M \leq 1$ by our choice of $n$, it must be that $n < M$.

We first write $\Sigma$ as a union of finitely many sets. Define 
\begin{align*}
\Lambda_1 &\coloneqq \{ (\sigma_j)_{j \in \mathbb{N}} \in \Sigma : \sigma_1 \geq e^M \},
\end{align*}
and for $k \geq 2$, define
\[
\Lambda_k \coloneqq \left\{ (\sigma_j)_{j \in \mathbb{N}} \in \Sigma : \prod_{j=1}^{k-1} \sigma_j < (k-1) e^M \text{ and } \prod_{j=1}^k \sigma_j \geq k e^M \right\}.
\]
We claim that $\Sigma = \bigcup_{k=1}^{n+1} \Lambda_k$. To prove the claim, we need to show that $\Sigma \subseteq \bigcup_{k=1}^{n+1} \Lambda_k$ since the reverse inclusion is obvious. Let $\sigma \coloneqq (\sigma_j)_{j \in \mathbb{N}} \in \Sigma$ and assume $\sigma \in \Sigma \setminus \bigcup_{k=1}^n \Lambda_k$. Then $\sigma_1 < e^M$ since $\sigma \not \in \Lambda_1$, $\sigma_1 \sigma_2 < 2 e^M$ since $\sigma \not \in \Lambda_2$, $\dotsc$, $\prod_{j=1}^{n-1} \sigma_j < (n-1) e^M$ since $\sigma \not \in \Lambda_{n-1}$, and $\prod_{j=1}^n \sigma_j < n e^M$ since $\sigma \not \in \Lambda_n$. Since we have $\prod_{j=1}^{n+1} \sigma_j \geq (n+1) e^M$ by \eqref{bound for product of n+1 digits}, it must be that $\sigma \in \Lambda_{n+1}$. Therefore, $\sigma \in \bigcup_{k=1}^{n+1} \Lambda_k$ and this proves the claim.

For each $1 \leq k \leq n+1$, our aim is to find a covering of $\Gamma (\Lambda_k)$ consisting of squares of side length $\varepsilon = 2 e^{-M}$ and to determine an upper bound, which we will denote by $a_k$, of the number of required squares.

Let $\sigma \coloneqq (\sigma_j)_{j \in \mathbb{N}} \in \Lambda_1$. Then $\varphi (\sigma) \leq 1/\sigma_1$ by definition of $\varphi$, and so $\varphi (\sigma) \in [0, e^{-M}]$ by definition of $\Lambda_1$. We know that $- 1/(\sigma_1 \sigma_2) \leq \mathcal{E}^*(\sigma) \leq 0$ from Theorem \ref{boundedness theorem} and its proof. Since $\sigma_1 \geq e^M$ and $\sigma_2 > \sigma_1$, it follows that
\begin{align*}
\big| \mathcal{E}^* (\sigma) \big|
&\leq \frac{1}{\sigma_1 \sigma_2} \leq \frac{1}{e^M (e^M+1)} < e^{-M} < \varepsilon.
\end{align*}
Hence, $\Gamma (\Lambda_1)$ can be covered by $a_1 \coloneqq 1$ square of side length $\varepsilon= 2e^{-M}$.

Let $k \in \{ 2, \dotsc, n+1 \}$. For every $\sigma \coloneqq (\sigma_j)_{j \in \mathbb{N}} \in \Lambda_k$, since $\prod_{j=1}^k \sigma_j \geq k e^M$, we have by Lemma \ref{epsilon cover} that
\[
\big| \varphi (\sigma) - \varphi (\sigma^{(k-1)}) \big| \leq \frac{1}{\sigma_1 \sigma_2 \dotsm \sigma_k} \leq e^{-M}
\quad \text{and} \quad
\big| \mathcal{E}^*(\sigma) - \mathcal{E}^*(\sigma^{(k-1)}) \big| \leq \frac{k-1}{\sigma_1 \sigma_2 \dotsm \sigma_k} \leq e^{-M}.
\]
This shows that for a fixed $\tau \coloneqq (\tau_j)_{j \in \mathbb{N}} \in \Sigma_{k-1}$, we can cover $\Gamma(\Lambda_k \cap \Upsilon_\tau)$ by one square of side length $2e^{-M} = \varepsilon$. Since $\prod_{j=1}^{k-1} \sigma_j < (k-1) e^M$ by definition of $\Lambda_k$, using Lemma \ref{refined bound lemma}, we see that at most
\[
a_k \coloneqq \frac{1}{(k-1)!} (k-1) e^M ( 2 + M + \log (k-1) )^{k-2}
\]
squares of side length $2 e^{-M} = \varepsilon$ are needed to cover $\Gamma (\Lambda_k)$.

Denote by $N_\varepsilon$ the smallest number of squares of side length $\varepsilon$ needed to cover $\Gamma(\Sigma)$. Clearly, $\Gamma(\Sigma) = \bigcup_{k=1}^{n+1} \Gamma(\Lambda_k)$. Then, by the discussion so far,
\[
N_\varepsilon \leq \sum_{k=1}^{n+1} a_k = 1 + \sum_{k=2}^{n+1} \frac{(k-1) e^M ( 2 + M + \log (k-1) )^{k-2}}{(k-1)!}.
\]
Now note that $a_1 < a_2 = e^M$, and for $2 \leq k \leq n$,
\begin{align*}
\frac{a_{k+1}}{a_k} = \frac{k}{k-1} \left( \frac{2+M+\log k}{2+M+\log (k-1)} \right)^{k-1} \frac{2+M+\log (k-1)}{k} > 1 \cdot 1^{k-1} \cdot \frac{M}{n} > 1,
\end{align*}
where the last inequality holds true since $M > n$. So $a_{n+1}> a_n > \dots > a_1$, and it follows that
\begin{align*}
N_\varepsilon \leq \sum_{k=1}^{n+1} a_{n+1} = (n+1) \frac{n e^M ( 2 + M + \log n )^{n-1}}{n!}.
\end{align*}
Recall that by our choice of $n$, we have $e^M \leq n!$ and $n < M$, and so
\[
N_\varepsilon \leq (n+1) n ( 2 + M + \log n )^{n-1} < (M+1)^2 ( 2 + M + \log M )^{n-1}.
\]
Now
\begin{align*}
\frac{\log N_\varepsilon}{\log (1/\varepsilon)} 
&\leq \frac{2\log (M+1) + (n-1) \log ( 2 + M + \log M )}{M - \log 2} \\
&= \frac{2\log (M+1)}{M - \log 2} + \frac{(n-1)\log M}{M - \log 2} + \frac{(n-1) \log ( (2 + \log M)/M + 1)}{M - \log 2},
\end{align*}
and we will estimate the upper limit of each of the three terms in the second line above. Clearly, the limit of the first term is $0$ as $M \to \infty$. For the second term, using \eqref{bounds for M}, we have
\[
\frac{(n-1) \log M}{M - \log 2} \leq \frac{(n-1) \log (n+1) + (n-1) \log (\log n)}{(n-1) \log (n-1)-(n-2) -\log 2} \to 1
\]
as $n \to \infty$. For the last term, notice that $\log ( (2 + \log M)/M + 1) \to 0$ as $M \to \infty$ and $n-1 < M-\log 2$, to deduce that the limit is $0$. Thus, since $\overline{G(I)} = \Gamma (\Sigma)$ by Lemma \ref{Gamma is a closure of G}, we finally obtain
\[
\ubdim \overline{G(I)} = \ubdim \Gamma (\Sigma) = \limsup_{\varepsilon \to 0} \frac{\log N_\varepsilon}{\log (1/\varepsilon)} \leq 1.
\]
\end{proof}

\begin{proof}[Proof of Theorem \ref{box dimension of graph}]
In view of Proposition \ref{hdim lbdim ubdim}, combining Lemmas \ref{lower box dimension} and \ref{upper box dimension} gives $\lbdim \overline{G(I)} = \ubdim \overline{G(I)} \allowbreak = 1$. Since taking closure of a set does not alter the upper and lower box-counting dimensions by \cite[Proposition~2.6]{Fal14}, it follows that $\lbdim G(I) = \ubdim G(I) = 1$ and therefore, we conclude that $\bdim G(I) = 1$.
\end{proof}

\begin{remark}
We point out that Lemma \ref{upper box dimension} gives an alternative proof of the upper bound part in Theorem \ref{Hausdorff dimension of graph}. In fact, due to Proposition \ref{hdim lbdim ubdim}, we have $\hdim \overline{G(I)} \leq \lbdim \overline{G(I)} \leq \ubdim \overline{G(I)}$ and, furthermore, $\ubdim \overline{G(I)} \leq 1$ by Lemma \ref{upper box dimension}. But then monotonicity of the Hausdorff dimension implies that $\hdim G(I) \leq 1$, which is the upper bound in Theorem \ref{Hausdorff dimension of graph}. 
\end{remark}

\subsection{The covering dimension of the graph of $\mathcal{E}(x)$}

The graph of $\mathcal{E} \colon I \to \mathbb{R}$ has the same Hausdorff dimension and box-counting dimension, both equaling one. In this subsection, we show that the covering dimension of the graph of $\mathcal{E}$ is zero, so that it is strictly smaller than the Hausdorff dimension.

\begin{theorem} \label{covering dimension theorem}
The graph of the error-sum function $\mathcal{E} \colon I \to \mathbb{R}$ has the covering dimension zero, i.e., $\covdim G(I) = 0$.
\end{theorem}

We say that a topological space $X$ is {\em totally separated} if for every pair of distinct points $x, y \in X$, there are disjoint open sets $U$ and $V$ such that $x \in U$, $y \in V$, and $X = U \cup V$. The following propositions will be used for the proof of the theorem.

\begin{proposition} [{\cite[Theorem~2.7.1]{Coo15}}] \label{covering dimension lemma}
Let $X$ be a non-empty compact Hausdorff space. Then $X$ is totally separated if and only if $\covdim X = 0$.
\end{proposition}

\begin{proposition} [{\cite[Theorem~1.8.3]{Coo15}}] \label{monotonicity lemma}
If $X$ is a metrizable space and $Y \subseteq X$, then $\covdim Y \leq \covdim X$.
\end{proposition}

The theorem is a consequence of the following lemma.

\begin{lemma} \label{totally separated lemma}
We have $\covdim \Gamma (\Sigma) = 0$.
\end{lemma}

\begin{proof}
Obviously, $\Gamma (\Sigma) \subseteq \mathbb{R}^2$ is non-empty and Hausdorff, and, furthermore, by Lemma \ref{properties of Gamma}(ii), it is compact. By Proposition \ref{covering dimension lemma}, it is sufficent to show that $\Gamma (\Sigma)$ is totally separated. To see this, first recall from Lemma \ref{properties of Gamma}(i) that $\Gamma \colon \Sigma \to \Gamma (\Sigma)$ is a homeomorphism. It is clear that $\mathbb{N}_\infty$ is totally separated, and so is its (countable) product $\mathbb{N}_\infty^\mathbb{N}$. It follows that $\Sigma \subseteq \mathbb{N}_\infty^\mathbb{N}$ is also totally separated. Hence its homeomorphic image $\Gamma (\Sigma)$ is totally separated. This proves the result.
\end{proof}

\begin{proof}[Proof of Theorem \ref{covering dimension theorem}]
On one hand, since $G(I) \neq \varnothing$, we have $\covdim G(I) \geq 0$ by \cite[Example~1.1.9]{Coo15}. On the other hand, since $G(I)$ is a subset of the metrizable space $\Gamma (\Sigma) \subseteq \mathbb{R}^2$, Proposition \ref{monotonicity lemma} and Lemma \ref{totally separated lemma} tell us that $\covdim G(I) \leq 0$. This completes the proof.
\end{proof}

\section*{Acknowledgements}

I wish to thank my advisor, Dr.\ Hanfeng Li, for helpful comments and suggestions, continuous support, and encouragement to work on the subject.

\end{document}